\documentclass[12pt]{amsart}
\usepackage[utf8]{inputenc}
\usepackage[margin=1in]{geometry}
\usepackage[margin=.5in,small]{caption}
\usepackage{graphicx,amssymb,url,adjustbox,setspace}
\usepackage[tracking,spacing,kerning]{microtype}
\usepackage[colorlinks,final]{hyperref}
\microtypecontext{spacing=nonfrench}

\title{Dessins d'enfants for analysts}
\date{\today}

\author{Vincent Beffara}
\address{UMPA, ENS Lyon and HCM, Bonn}
\email{vbeffara@ens-lyon.fr}
\urladdr{http://perso.ens-lyon.fr/vincent.beffara/}
\thanks{This  work  was  started  during  a stay  at  the  Max  Planck
  Institute and  the Hausdorff  Center in Bonn,  and completed  at the
  Newton  Institute in  Cambridge;  support of  these institutions  is
  gratefully acknowledged.}

\newtheorem{theorem}{Theorem}
\newtheorem{lemma}[theorem]{Lemma}
\theoremstyle{definition}
\newtheorem*{definition}{Definition}
\theoremstyle{remark}
\newtheorem{remark}{Remark}

\let\leq=\leqslant \let\geq=\geqslant

\begin{document}
\begin{abstract}
  We present an  algorithmic way of exactly  computing Belyi functions
  for  hypermaps and  triangulations  in  genus $0$  or  $1$, and  the
  associated  dessins,   based  on  a  numerical   iterative  approach
  initialized from  a circle packing combined  with subsequent lattice
  reduction. The main  advantage compared to previous  methods is that
  it is applicable to much larger graphs; we use very little algebraic
  geometry,  and  aim  for  this  paper to  be  as  self-contained  as
  possible.
\end{abstract}

\maketitle

\centerline{\includegraphics[width=.9\linewidth]{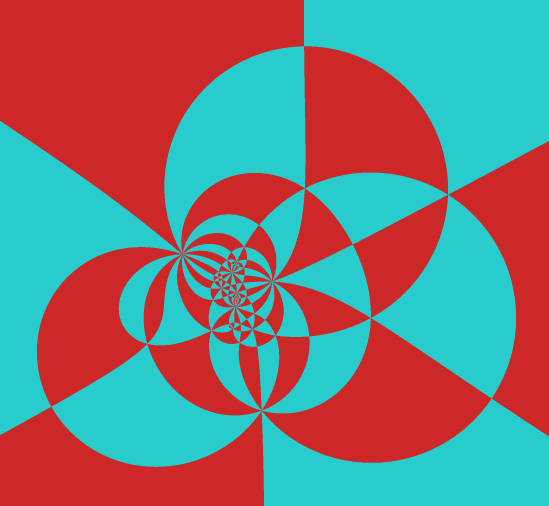}}

\section*{Introduction}

Let $T$ be a triangulation of  the sphere $\mathbb S^2$, \emph{i.e.} a
planar graph embedded in  the sphere in such a way  that all the faces
it delimits  have degree $3$.  From $T$,  one can construct  a complex
structure  on  the sphere  by  gluing  together equilateral  triangles
according to the  combinatorics of $T$; this leads to  a manifold with
conical singularities at the locations of the vertices of $T$, and the
Riemann uniformization  theorem ensures  the existence of  a conformal
bijection  $\Phi_T$  between  this  surface  and  the  Riemann  sphere
$\mathbb C^\ast$. The  image of $T$ by $\Phi_T$ is  a triangulation of
the  Riemann sphere  with  the  same combinatorics  as  $T$ and  edges
embedded as analytic curves, which is well-defined up to the action of
the Möbius group.

While  this   construction  is  quite  simple,   \emph{computing}  the
embedding  explicitly is  very difficult  in general,  as soon  as the
number of vertices of $T$ is not  very small. One of the goals of this
paper  is to  describe  an  algorithm to  do  it  automatically, in  a
``semi-numerical''  way  (that  is,   going  through  an  approximate,
numerical stage to obtain an exact, explicit outcome).

Assume in addition that all the vertices of $T$ have even degree. This
is equivalent to two other properties of the triangulation: first, its
faces can be  partitioned into two sets  in such a way that  a face in
one set is  only adjacent to faces  in the other set;  this is usually
done graphically by coloring one half of the faces white and the other
half black. Second, the vertices of  $T$ can be partitioned into three
classes in  such a way that  along the boundary of  each face, exactly
one vertex of each class appears.

Choosing the triangles used to build our Riemann surface to be the two
hemispheres of the Riemann sphere separated by the extended real line,
with vertices at $0$, $1$ and $\infty$, from the data of $T$ we obtain
a  covering  $\pi_T :  \mathbb  C^\ast  \to  \mathbb C^\ast$  that  is
ramified  only  above $\{0,1,\infty\}$  (such  a  map  is known  as  a
\emph{Belyi function}), and  the embedding of $T$  in $\mathbb C^\ast$
can be seen as the preimage of $\mathbb R \cup \{\infty\}$ by $\pi_T$,
the three  classes of  vertices being  the preimages  of $0$,  $1$ and
$\infty$  respectively. So,  the question  of computing  the embedding
turns into that of computing $\pi_T$.

The covering $\pi_T$ is meromorphic, and  in the case of the sphere it
means that it  has to be a rational function;  computing it then means
computing the location of its zeros and its poles, or equivalently the
coefficients of its  numerator and denominator. It  is always possible
to choose the embedding (\emph{via} a Möbius transformation) in such a
way that all these numbers are  algebraic, and what we are looking for
is their specification  as roots of explicit  polynomials with integer
coefficients.  This has  been  done  in many  cases  before, cf.\  for
example  \cite{BZ:arbre,BM:K3,SS:belyi,Z:belyi}, tracing  back to  the
original  work  of  Klein  on   the  icosahedron  and  its  link  with
fifth-degree  equations~\cite{klein:ikosaeder}; a  recent work  with a
similar    goal    as    ours,    but    very    different    methods,
is~\cite{BBGK:hurwitz}.

\bigskip

The same  construction can  be performed from  a triangulation  of any
orientable surface, and  leads to a Riemann surface $M_T$  of the same
genus       together       with        a       ramified       covering
$\pi_T : M_T \to \mathbb C^\ast$; this  surface in turn can be seen as
a (smooth irreducible projective) algebraic curve, and Belyi's theorem
(see below for a precise statement) ensures that this curve is defined
over $\bar{\mathbb Q}$.  A similar goal as before is  then to identify
this algebraic curve explicitly, by specifying the coefficients of its
equation as roots of explicit integer polynomials, as well as those of
the covering.

In this  paper, besides the  sphere we  will only discuss  the simpler
case of genus $1$, where the Riemann surface built from $T$ can always
be            uniformized            into           a            torus
$\mathbb  T =  \mathbb C  / (\mathbb  Z +  \tau \mathbb  Z)$ for  some
modulus $\tau$ (depending on $T$) in the complex upper-half plane. The
previous problem  becomes the exact  computation of $\tau$ and  of the
covering $\pi_T :  \mathbb T \to \mathbb C^*$. Lifting  $\pi_T$ to the
universal cover of $\mathbb T$, this  means that we are looking for an
elliptic function rather  than for a rational one as  above, but apart
from that the situation remains very  similar in principle, in that we
are  looking  for  a  finite  collection  of  algebraic  numbers.  The
higher-genus situation  should be amenable  to a similar  treatment as
the one  we describe  here, but  the implementation  would have  to be
significantly more complex.

\bigskip

The usual  ways that  the question has  been addressed  previously are
mostly  algebraic in  nature:  one  can write  the  conditions that  a
function has to satisfy in order to be a Belyi function as a system of
polynomial equations, and try to solve it exactly from the start. This
works well for small examples,  but the combinatorial complexity grows
very quickly and  even moderately large cases,  especially in positive
genus,  are beyond  the reach  of computer  algebra systems.  A recent
proposal  by  Bartholdi  et.\ al.\  \cite{BBGK:hurwitz}  for  instance
quotes  a computing  time of  15 minutes  for a  triangulation of  the
sphere with 15 vertices.

Here instead  we base our approach  on a numerical point  of view, and
the central step  is an implementation of the Newton  algorithm to get
an arbitrary precision  approximation of the algebraic  numbers we are
looking for,  initialized with  a configuration obtained  using circle
packings; this can then be combined with a lattice-reduction algorithm
to identify the corresponding integer polynomials. For comparison, for
the same example as used in~\cite{BBGK:hurwitz}, getting 100 digits of
precision for all  parameters (which is more than enough  to obtain an
exact solution) takes less than a second.

Our  initial motivation  for this  work was  related to  probabilistic
conjectures on  the conformal structure  of large random  planar maps;
testing these conjectures numerically  without computing the embedding
exactly is difficult,  because in this setup  approximations come from
many sources: from the randomness of the map, from the fact that it is
of finite size,  and also from the embedding  approximation itself, so
determining  the   embedding  exactly  is  of   practical  importance.
Computing  the  embedding  algebraically  being beyond  reach  of  the
previous methods for the numbers of vertices that are relevant in this
setup, we were led to look for  a more numerical approach which is the
focus of  this paper; we will  apply it to random  triangulations in a
subsequent work.

\bigskip

The   remainder    of   this   paper   is    organized   as   follows:
section~\ref{sec:def} contains  the formal definitions of  the objects
that  we  are  using,  section~\ref{sec:belyi} gives  a  statement  of
Belyi's   theorem  and   is  there   mostly  to   keep  the   argument
self-contained,    and   sections~\ref{sec:num}    and~\ref{sec:torus}
describe our approach  in detail and provide proofs  of convergence in
the  cases of  genus $0$  and $1$,  respectively. Section~\ref{sec:ex}
then gives a few examples of application.

The source code  implementing the algorithms, and used  to produce all
the examples and pictures in this  paper, is publicly available at the
following address:

\bigskip

\centerline{\url{http://github.com/vbeffara/Simulations}}

\section{Definitions}
\label{sec:def}

\subsection{Maps and hypermaps}
\label{sec:maps}

Our starting point will always be a graph drawn on a given topological
surface,  but  seen  as  a combinatorial  object,  \emph{i.e.}  up  to
homeomorphisms  of  the surface.  All  the  following definitions  are
classical, but we still give them for sake of self-containedness. Much
more  can   be  found  for   instance  is   the  book  of   Lando  and
Zvonkin~\cite{LZ:graphs}.

\begin{definition}
  A (finite) \emph{graph} is a pair $G=(V,E)$ where $V$ is seen as the
  set of \emph{vertices} of $G$ and $E  \subset V^2$ as the set of its
  \emph{edges}. We will always  consider \emph{undirected} graphs, for
  which $E$ is  symmetric: $(x,y) \in E \iff (y,x)  \in E$, and assume
  the  absence  of \emph{loops},  \emph{i.e.}  of  edges of  the  form
  $(x,x)$ linking a vertex to itself.
\end{definition}

\begin{definition}
  An \emph{embedding} of an undirected  graph $G=(V,E)$ in an oriented
  surface $M$ is the data of a collection $(x_v)_{v\in V}$ of pairwise
  distinct      points     of      $M$,      and     a      collection
  $(\gamma_{x,y})_{(x,y)\in E}$  of continuous  simple curves  on $M$,
  with  $\gamma_{x,y}  : [0,1]  \to  M$  satisfying $\gamma(0)=x$  and
  $\gamma(1)=y$        and        the        symmetry        condition
  $\gamma_{x,y}(t)   =  \gamma_{y,x}(1-t)$,   and  such   that  curves
  corresponding to distinct edges have disjoint images except possibly
  for  their  endpoints.  The  embedding  is  \emph{proper}  if  every
  connected component of the complement of the union of all the images
  of  the  $\gamma_{x,y}$  (in   other  words,  each  \emph{face})  is
  homeomorphic to a disk.
\end{definition}

\begin{definition}
  A \emph{map} is an equivalence class of proper embeddings of a graph
  $G$ in  a surface $M$, where  two embeddings are identified  if they
  are conjugated  by a homeomorphism of  $M$. When $M$ is  the Riemann
  sphere $\mathbb  C^\ast$, we will  speak about a  \emph{planar map};
  when $M$ is a two-dimensional torus, about a \emph{toroidal map}. If
  each of the faces has exactly  $3$ edges along its boundary, the map
  is called a \emph{triangulation} of $M$.
\end{definition}

\begin{figure}[h]
  \centering
  \includegraphics[width=.49\linewidth]{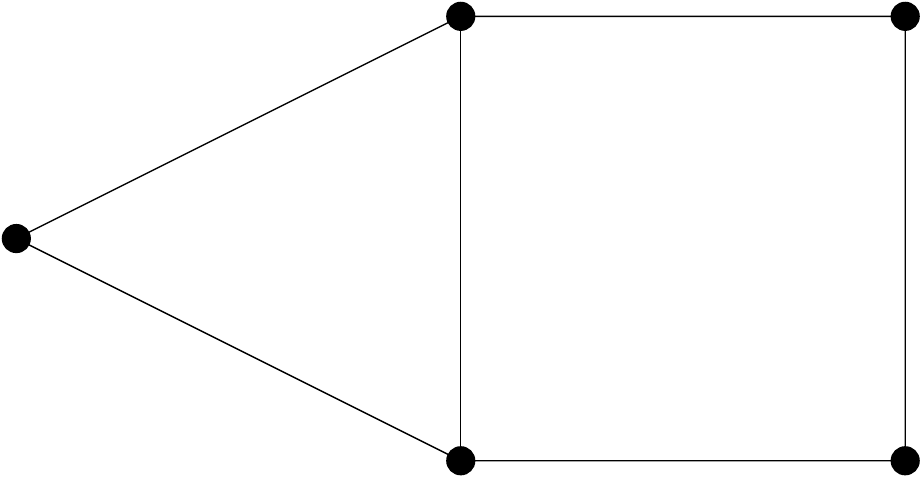}
  \caption{A planar  map with  $5$ vertices, $6$  edges and  $3$ faces
    (including the outer one in this representation).}
  \label{fig:map}
\end{figure}

A map  is in fact a  combinatorial structure, in the  sense that there
are only finitely many maps consisting in embeddings of a given finite
graph  in  a given  surface;  it  can  be  specified by  ordering  the
neighbors of each of the vertices of the graph in a cyclic order. This
remark allows for the definition of  the \emph{dual map} of a map: its
(dual) vertices are  in bijection with the faces of  the map; two dual
vertices will be  declared adjacent if the corresponding  faces of the
initial map share  an edge, and the neighbor ordering  in the dual map
is then given by the cyclic order around the corresponding face of the
primal map.

\bigskip

A very related structure,  which makes computer implementations easier
to manage in practice and is of algebraic relevance, is the following:

\begin{definition}
  Given a  positive integer $n$,  a \emph{hypermap}  of size $n$  is a
  triple  $(\sigma,\alpha,\varphi)$ of  permutations of  a set  of $n$
  elements       satisfying      the       compatibility      relation
  $\sigma \alpha \varphi = \mathrm{id}$.
\end{definition}

This can  be seen as  a generalization of  the notion of  map: indeed,
restricting  to  the  case  where  the  cycles  of  $\alpha$  are  all
transpositions, one  can interpret the  domain of the  permutations as
the set of  all the half-edges of  the map, the cycles  of $\sigma$ as
its  vertices  (reading the  half-edges  incident  to that  vertex  in
counterclockwise order),  those of  $\alpha$ as  its edges  (which are
pairs of half-edges) and those of  $\varphi$ as its faces (reading the
half-edges  emanating from  the  vertices along  that  face in  direct
order) ---  see Figure~\ref{fig:map_h} for  an example, which  is much
clearer than any formal description would be.

\begin{figure}[h]
  \centering
  \includegraphics[width=.49\linewidth]{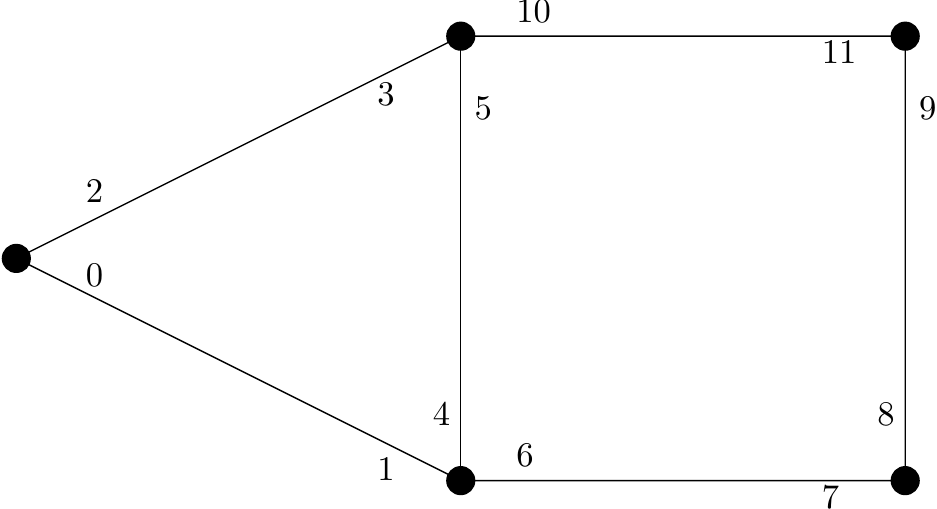}
  \caption{Numbering   of    the   half-edges   of   the    map   from
    Figure~\ref{fig:map}  to obtain  its  hypermap representation:  in
    this                case,                 we                obtain
    $\sigma   =   (0\,2)   (1\,6\,4)   (3\,5\,10)   (7\,8)   (9\,11)$,
    $\alpha  =  (0\,1)  (2\,3)  (4\,5)  (6\,7)  (8\,9)  (10\,11)$  and
    $\varphi  = (0\,4\,3)  (1\,2\,10\,9\,7) (5\,6\,8\,11)$.  Following
    the usual convention, the label of each half-edge is to its left.}
  \label{fig:map_h}
\end{figure}

The  maps   given  as  examples   in  this  paper  (for   instance  in
Appendix~\ref{sec:ex})   are  described   in  this   form,  with   the
$n$-element   set  chosen   as  $\{0,   \ldots,  n-1\}$.   The  triple
corresponding       to      the       dual      map       is      then
$(\varphi^{-1},  \alpha^{-1},  \sigma^{-1})$  --- and  obviously  here
$\alpha^{-1} = \alpha$, but this way of denoting it makes the notation
clearer. Conversely, hypermaps  can also be seen as  a particular case
of maps:

\begin{definition}
  A  graph  is  called  \emph{bipartite}  if its  vertex  set  can  be
  partitioned  into two  disjoint subsets,  in  such a  way that  each
  vertex in one of them is only adjacent to vertices in the other one.
  A map is bipartite if the  associated graph is bipartite; if the map
  is proper, this  is equivalent to saying that its  faces all have an
  even number of edges along their boundaries.
\end{definition}

It  is customary  to  refer to  the  partition of  the  vertices of  a
bipartite map into \emph{black} and  \emph{white} vertices. Such a map
can  then be  encoded  as a  hypermap, where  the  cycles of  $\sigma$
(resp.\ $\alpha$, $\varphi$) correspond  to the black vertices (resp.\
white vertices,  faces). In  this case, the  cycles of  $\varphi$ have
a length equal to a half of the number of edges on the boundary of the
corresponding faces of the map. If  every white vertex has degree $2$,
then the black vertices themselves form a graph with the same hypermap
representation.

One can compose  these two constructions, starting from  a map, seeing
it as  a hypermap where  the cycles of  $\alpha$ have length  $2$, and
then seeing this hypermap as a  bipartite map where the black vertices
correspond to the vertices of the initial map and the white ones (with
degree  $2$) to  its edges.  More graphically,  this is  equivalent to
adding  one vertex  on  each edge  of  the initial  map,  as shown  in
Figure~\ref{fig:map_b}.

\begin{figure}[h]
  \centering
  \includegraphics[width=.49\linewidth]{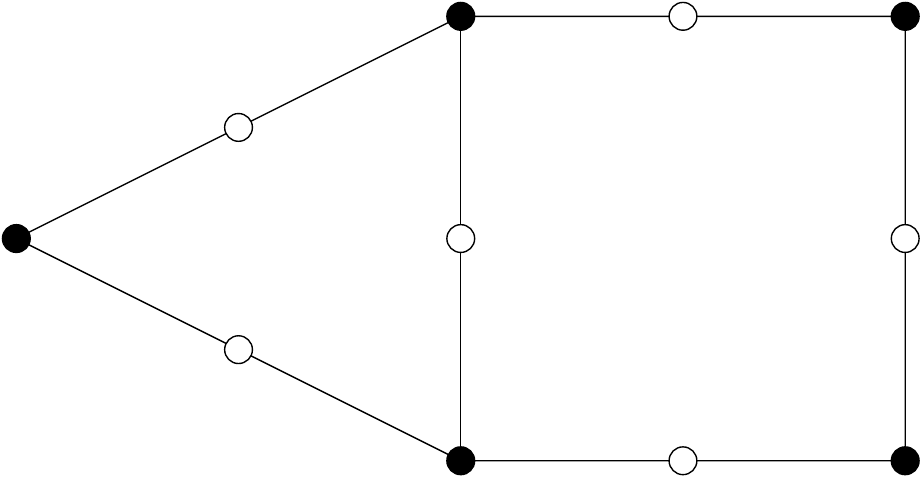}
  \caption{The  bipartite   map  corresponding  to  the   hypermap  in
    Figure~\ref{fig:map_h}.}
  \label{fig:map_b}
\end{figure}

The  last  identification  that  we   will  need  is  with  particular
triangulations:

\begin{definition}
  A  triangulation of  a surface  is called  \emph{tripartite} if  its
  vertex set can  be partitioned into $3$ disjoint subsets,  in such a
  way that each of  its faces has one vertex of  each of these subsets
  along its boundary.  This is equivalent to saying that  its dual map
  is bipartite, and that all its vertices have even degree.
\end{definition}

The three  vertex sets will  be represented  here as black,  white and
red; in~\cite{LZ:graphs},  vertices of the third  kind are represented
as  asterisks.   To  each  tripartite  triangulation   correspond  $3$
bipartite maps,  obtained by keeping  the vertices  in two out  of the
three  subsets  in  the  definition and  the  edges  connecting  them.
Conversely, every proper  bipartite map can be obtained  this way, and
another way of stating this is  that every proper bipartite map can be
completed  into a  tripartite  triangulation of  the  same surface  by
adding a  vertex inside each  face and  connecting it to  the original
vertices on the boundary of that face.

Combining  the previous  remarks,  every (proper)  planar  map can  be
refined into a  tripartite triangulation by adding one  vertex on each
of its edges and one vertex  inside each of its faces, connecting them
in the  natural way.  If $(\sigma, \alpha,  \varphi)$ is  the hypermap
corresponding to  the initial map,  then the  cycles of each  of these
permutations are  in bijection with the  vertices in one of  the three
subsets  in the  partition of  the tripartite  triangulation. We  will
refer to this triangulation as the \emph{tripartite refinement} of the
map   ---    see   Figure~\ref{fig:map_t}.   In   the    diagrams   of
Section~\ref{sec:ex},  the  edges incident  to  red  vertices are  not
represented in order to make the combinatorics more readable.

\begin{figure}[h]
  \centering
  \includegraphics[width=.67\linewidth]{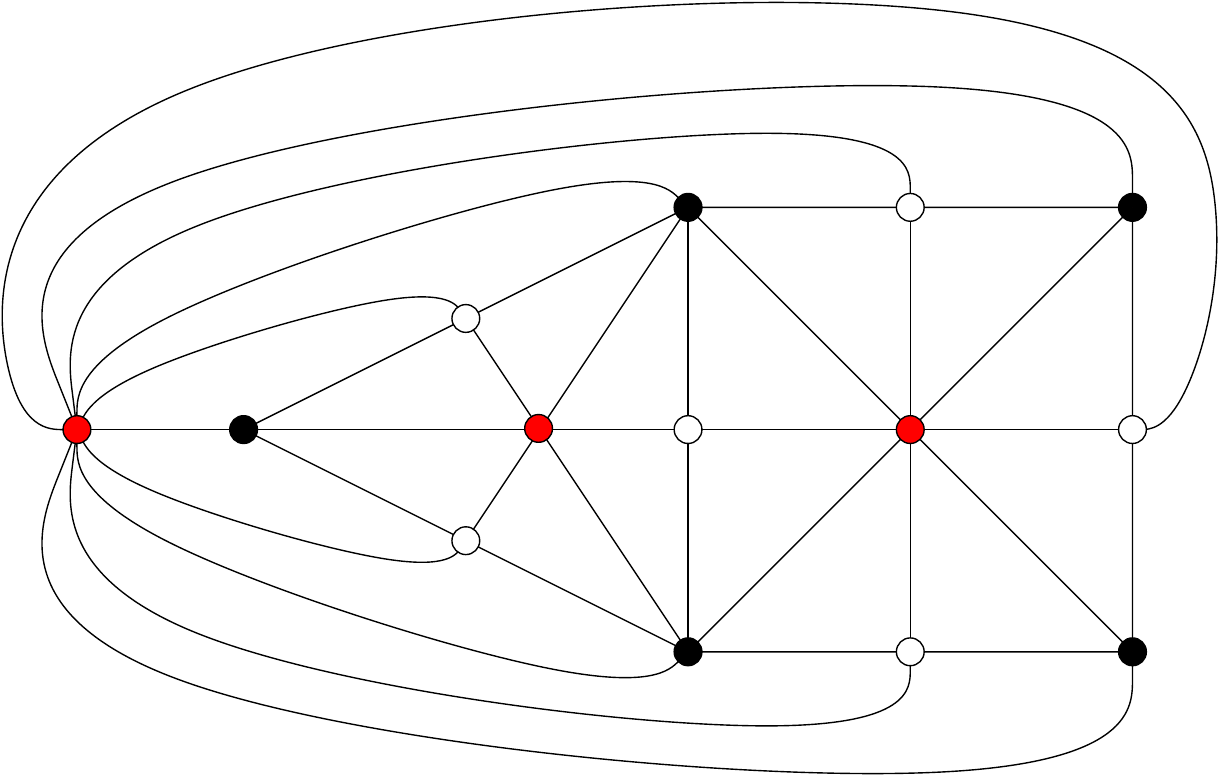}
  \caption{The  tripartite  refinement  of  to  the  planar  map  from
    Figure~\ref{fig:map}. One  can check that  the degree of  each red
    vertex is equal to twice the  number of edges of the corresponding
    face in the initial map.}
  \label{fig:map_t}
\end{figure}

\subsection{Constellations}
\label{sec:const}

Our main objects  of study in this paper are  holomorphic maps defined
between Riemann  surfaces, but it is  convenient as a way  of encoding
Belyi maps to introduce a more  discrete kind of object. In this whole
section, $M$ will denote a fixed Riemann surface. Our first definition
is that of a decorated point of $M$:

\begin{definition}
  A \emph{star} on  $M$ is a pair  $\mathbf z = (z,d)$  where $z\in M$
  (the   \emph{location})  and   $d$  is   a  positive   integer  (the
  \emph{multiplicity}).
\end{definition}

The idea  is that a  star will  specify the behavior  of a map  at its
location (typically $d$ would be the degree of ramification at $z$). A
polynomial can be  given as a finite sequence of  stars located at its
roots, and  a rational fraction as  two such sequences (for  the roots
and the poles). To match later with Belyi's theorem, we will need just
a little more information:

\begin{definition}
  A \emph{constellation}  on $M$  is a triple  $\mathcal C  = (Z,P,O)$
  where   $Z=\{\mathbf   z_i\}_{1   \leq   i   \leq   n_z}$,
  $P=\{\mathbf    p_i\}_{1   \leq    i   \leq    n_p}$   and
  $O=\{\mathbf o_i\}_{1  \leq i \leq n_o}$  are three finite
  (potentially  empty)  sets   of  stars,  which  we   will  refer  to
  respectively as  the \emph{zeros},  \emph{poles} and  \emph{ones} of
  $\mathcal  C$.  A  constellation  is  \emph{non-degenerate}  if  the
  locations of all involved stars are pairwise distinct (which we will
  always implicitly assume except otherwise mentioned).
\end{definition}

If $f: M \to \mathbb C^\ast$  is a holomorphic covering of the Riemann
sphere,     it     comes     naturally    with     a     constellation
$\mathcal C_f = (Z_f,P_f,O_f)$ where $Z_f$ lists the preimages of $0$,
$P_f$ those of $\infty$ and $O_f$ those of $1$, and the multiplicities
are the  corresponding orders  of ramification. Under  the assumptions
that  $f$ only  ramifies  over  $\{0,1,\infty\}$, $f$  is  known as  a
\emph{Belyi function}, and $\mathcal  C_f$ characterizes $f$ uniquely;
this relation can be made much more  explicit in a few cases, which we
describe now.

\subsubsection{Polynomials on $\mathbb C$}
\label{sec:poly}

If $f \in \mathbb C[X]$ is  a polynomial with complex coefficients and
degree $d$,  it is specified uniquely  by the collection of  its roots
and  their  multiplicities,  and   its  leading  coefficient:  if  its
constellation                is                given                as
$$\mathcal  C_f =  (\{(z_i,d_i)\}, \{(\infty,d)\},  \{(o_i,d'_i)\})$$
(with    $\sum    d_i=d$)   then    $f$    can    be   recovered    as
$$f(z) =  \lambda \prod (z-z_i)^{d_i}$$  where $\lambda$ is  chosen to
make the  value of $f$ equal  to $1$ at  the points of $O_f$.  In this
setup, $f$ is a Belyi function if  and only if $\sum d'_i=d$, in which
case $$f(z)-1  = \lambda  \prod (z-o_i)^{d'_i} \quad  \text{and} \quad
f'(z) = \lambda d  \prod (z-z_i)^{d_i-1} \prod (z-o_i)^{d'_i-1}.$$ The
polynomial $f$ is then known as a \emph{Shabat polynomial}.

\subsubsection{Rational fractions on $\mathbb C$}
\label{sec:rat}

Let $f$ now be a (ramified)  covering of the Riemann sphere by itself:
it has to be a rational fraction with complex coefficients, and can be
written as $f(z)=p(z)/q(z)$ where $p$  (resp.\ $q$) is a polynomial of
degree  $d_1$   (resp.\  $d_2$).  Again  $Z_f   =  \{(z_i,d_i)\}$  and
$P_f = \{(p_i,d'_i)\}$ can be written explicitly in terms of the roots
of $p$  and $q$  respectively, and  $\infty$ will  occur in  $Z_f$, in
$P_f$ or  in neither  according to  the relative  values of  $d_1$ and
$d_2$. $f$ can then be similarly reconstructed as
\begin{equation}
  \label{eq:rec0}
  f(z)    =    \lambda     \frac    {\prod    (z-z_i)^{d_i}}    {\prod
    (z-p_i)^{d'_i}}
\end{equation}
(where the products omit the potential  term at $\infty$, and where as
before $\lambda$ is chosen so that  $f$ takes value $1$ at the $o_i$).
Once more,  the constellation  $\mathcal C_f$  corresponds to  a Belyi
function if  and only if the  sums of multiplicities are  the same for
$Z_f$, $P_f$ and $O_f$.

\subsubsection{Functions on a complex torus}
\label{sec:ellip}

The case of genus $1$ can also  be made quite explicit. We will assume
in this whole  section that $M$ is the complex  torus with periods $1$
and     $\tau$     where     $\Im(\tau)>0$,     in     other     words
$$M  = \mathbb  C /  (\mathbb Z  + \tau  \mathbb Z).$$  We will  abuse
notation by writing  the locations of stars as  complex numbers. Here,
$M$ comes with a natural covering $\pi :  \mathbb C \to M$ and if $f :
M \to  \mathbb C^*$  is holomorphic,  it can be  lifted as  a periodic
meromorphic  function  $\hat  f  :  \mathbb C  \to  \mathbb  C^*$  (an
\emph{elliptic function}).

The situation  is a little more  rigid than before, in  the sense that
some relations are  automatically satisfied by the zeros  and poles of
$f$.     More     specifically,     if     $Z_f=\{(z_i,d_i)\}$     and
$P_f=\{(p_i,d'_i)\}$ then
\begin{equation}
  \label{eq:sum0}
  \sum d_i = \sum d'_i \quad \text{and} \quad \sum d_i z_i \equiv \sum
  d'_i p_i
\end{equation}
(the second  relation being meant mod  $\mathbb Z + \tau  \mathbb Z$).
The reconstruction of $f$ from these  can be made in a similar fashion
as  in the  case of  rational fractions:  if the  $p_i$ and  $z_i$ are
chosen in such a way that $\sum d_i z_i = \sum d'_i p_i$ (this time as
complex  numbers  in  $\mathbb  C$),  then  $f$  can  be  obtained  as
\begin{equation}
  \label{eq:rec1}
  f(z)   =    \lambda   \frac   {\prod    \zeta(z-z_i)^{d_i}}   {\prod
    \zeta(z-p_i)^{d'_i}}
\end{equation}
in terms of the Weierstrass  function $\zeta$ with (quasi-)periods $1$
and  $\tau$.  It is  also  the  case that  $f$  can  be written  as  a
polynomial in the Weierstrass function $\wp$ with the same periods and
its derivative $\wp'$,  or equivalently as a rational  function on the
associated elliptic curve,  but although that is the  ``right'' way to
reconstruct $f$  in terms of  algebraic geometry, the  coefficients of
this polynomial  depend on  the constellation  in a  way that  is less
transparent and does not exhibit the same similarity with the previous
two cases.

\subsubsection{A few additional remarks}
\label{sec:constrk}

In the case of Belyi functions, it is always the case that the sums of
multiplicities is the same for $Z_f$,  $P_f$ and $O_f$, being equal to
the  degree of  the covering;  and  the genus  of the  surface $M$  is
related to these multiplicities by  the Riemann-Hurwitz formula. It is
therefore natural to define the following:

\begin{definition}
  A  constellation $\mathcal  C =  (Z,P,O)$ is  \emph{balanced} if  it
  satisfies
  $$\sum_{(w,d)\in Z} d = \sum_{(w,d)\in P} d = \sum_{(w,d)\in O} d =:
  N(\mathcal                                                     C).$$
  If $\mathcal C$ is balanced, its \emph{genus} is given by
  $$g(\mathcal C) := 1 + \frac {N (\mathcal C) - \#(Z \cup P \cup O)}
  2.$$
\end{definition}

With these notations, if $f : M  \to \mathbb C^*$ is a Belyi function,
then its  constellation $\mathcal  C_f$ is balanced  and its  genus it
that of $M$. As we saw previously, if  this genus is $0$ or $1$, it is
then  possible  to  explicitly  reconstruct  $f$  from  $\mathcal  C$.
Besides, if $(w,d) \in O_f$, then $d$ gives the degree of ramification
of $f$ at $w$. This can be used in the other direction to characterize
the constellations corresponding to a Belyi function:

\begin{definition}
  Let $\mathcal C = (Z,P,O)$ be  a balanced constellation of genus $0$
  or  $1$:  we will  say  that  it  is  \emph{exact} if  there  exists
  $\lambda \in \mathbb C$ satisfying the following conditions. Let $f$
  be defined  according to  either \eqref{eq:rec0}  or \eqref{eq:rec1}
  (depending  on the  genus  --- and  assume  that \eqref{eq:sum0}  is
  satisfied in the case of genus $1$). Then for every $(w,d)\in O$:
  \begin{itemize}
  \item $f(w)=1$;
  \item for every $k\in\{1,\ldots,d-1\}$, $f^{(k)}(w)=0$.
  \end{itemize}
  Such a  value of $\lambda$ is  obviously unique, and will  be called
  the \emph{canonical normalization} of $\mathcal C$.
\end{definition}

\subsection{Drawing a constellation}
\label{sec:draw}

Let $M$ be of genus $0$ or $1$  and $\mathcal C = (Z,P,O)$ be an exact
constellation  on  $M$;  let  $f  :  M  \to  \mathbb  C^\ast$  be  the
corresponding covering. This setup allows to draw two natural, related
structures on $M$.

\subsubsection{As a triangulation}
\label{sec:tri}

\begin{figure}[h]
  \centering
  \includegraphics[width=.8\linewidth]{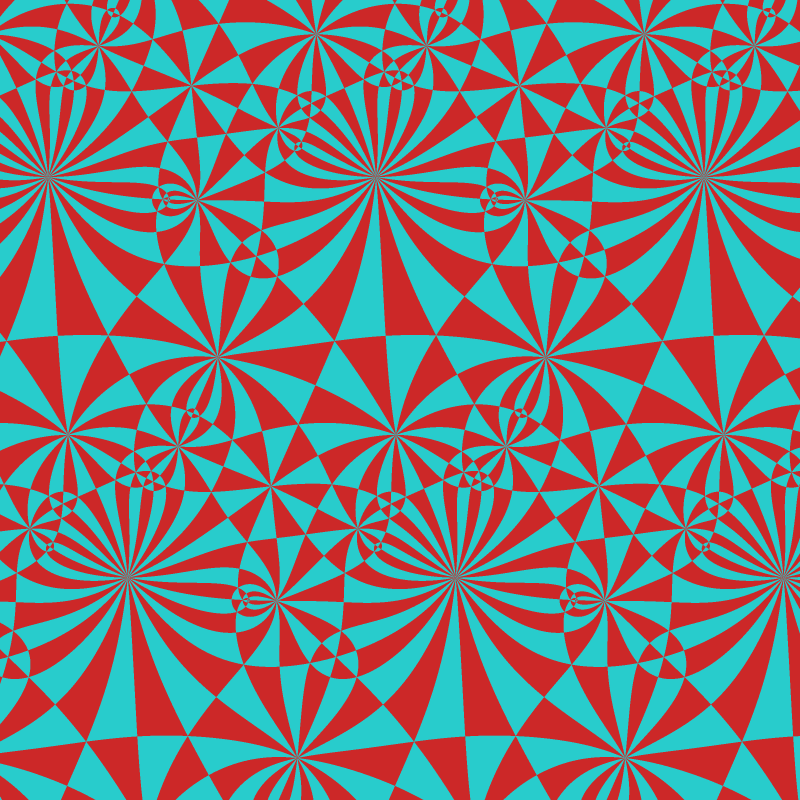}
  \caption{Uniformization   of  a   periodic  triangulation   of  $90$
    vertices. The color corresponds to  the sign of the imaginary part
    of the covering map.}
  \label{fig:random}
\end{figure}

First, one  can look  at the  tripartite map on  $M$ with  vertices at
$Z \cup  P \cup  O$, and  edges given  as the  preimages of  the three
intervals $(-\infty,0)$,  $(0,1)$ and $(1,+\infty)$ of  $\mathbb R$ by
$f$. This is a  triangulation of $M$, and its dual  is a bipartite map
with all vertices of degree $3$.  Most of the pictures in this article
are drawn in this setup (see \emph{e.g.} the first page for an example
in genus $0$, or Figure~\ref{fig:random}  in genus $1$), and the faces
are colored  according to  the sign  of the imaginary  part of  $f$ on
them.

\subsubsection{As a hypermap}
\label{sec:hyp}

Another way to draw $\mathcal C$  is as a hypermap, where $Z$ consists
in the  black vertices, $P$ in  the white vertices, and  the edges are
the preimages of the segment $[0,1]$  by $f$. The combinatorics of the
hypermap  (the triple  of permutations  $(\sigma,\alpha,\phi)$ in  the
language  of  the previous  section)  is  equivalent  to that  of  the
triangulation   above,   and   is   completely   determined   by   the
constellation. It is  this particular way of drawing  this hypermap on
$M$ which  is usually  referred to as  a \emph{dessin  d'enfant}. Even
though  the information  is theoretically  the same  as before,  it is
computationally  more   difficult  to  obtain   visually  satisfactory
pictures this way.

\section{Belyi's theorem}
\label{sec:belyi}

We give in  this section an extremely limited  introduction to Belyi's
theorem, for  the benefit  of readers  who are  not familiar  with the
topic. For much more, we refer to~\cite{Girondo:2012uw} and references
therein.

\subsection{Preliminary remarks}
\label{sec:prel}

For fixed balanced  multiplicities, the space of  constellations is of
complex dimension $\#(Z \cup P \cup O)$. The number of equations to be
satisfied for  a constellation $\mathcal C$  to be exact is  counted a
bit differently depending on the genus:
\begin{itemize}
\item  If $g=0$,  it  is  the number  of  conditions  coming from  the
  definition, which is equal to $N(\mathcal C) - 1$;
\item If  $g=1$, the relation  \eqref{eq:sum0} needs to be  counted as
  well and the total number is $N(\mathcal C)$.
\end{itemize}
The difference $\delta$ between the number of variables and the number
of conditions is then  equal to $3$ in genus $0$, and  to $0$ in genus
$1$; in particular it depends only on the genus.

Before we can make use of this computation, one remark is in order. In
the general  setting, automorphisms  of $M$  act in  a natural  way on
constellations   on  $M$   by  mapping   locations  while   preserving
multiplicities.  This preserves  the property  of being  exact; it  is
natural  to   identify  two  constellations  conjugated   by  such  an
automorphism, or at least to classify them up to automorphism.

In genus  $0$, the group  of automorphisms  of the Riemann  sphere has
complex dimension $3$, which is equal  to $\delta$. This means that we
should   expect  the   number   of  exact   constellations  of   given
multiplicities, up to automorphisms of $M$, to be finite.

In genus $1$, the group of  automorphisms of $M$ has complex dimension
$1$ (in the generic case it consists purely of translations), which is
one more than $\delta$. This means  that in the general case we should
expect   the  existence   of   no  exact   constellation  with   given
multiplicities,  but that  for finitely  many values  of $\tau$  there
should exist finitely many exact constellations.

Of course, all the preceding remarks  are at the heuristic level and
making   them  formal   would  imply   controlling  degeneracies   and
genericity, which would likely be quite difficult to do at this point.
Nevertheless,  as what  follows  will  show, they  do  give the  right
predictions.

\subsection{Analytic statement}
\label{sec:belyi_a}

Let $T$ be a map with triangular faces, of genus $0$ or $1$. It can be
made  into a  Riemann  surface $M_T$  by  gluing together  equilateral
triangles  according to  its combinatorics;  and this  surface can  be
uniformized to either  the Riemann sphere, or to a  complex torus with
periods $1$ and $\tau_T$, where $\tau_T$ is uniquely determined (up to
$SL_2(\mathbb Z)$  action) by $T$. It  is a natural question,  and the
origin of the work presented here, to in the latter case determine the
value  of $\tau_T$  from  the discrete  data of  the  map $T$,  either
exactly or at least numerically (which, if done with enough precision,
suffices to obtain an exact value).

If the  vertices of $T$ all  have even degree (which  is in particular
the case  if $T$  is associated to  a hypermap), they  can be  seen as
stars with multiplicity half of their degree and it is always possible
to split them into three disjoint sets  (or types) $Z$, $P$ and $O$ in
such a way that each of the faces contains one vertex in each of these
sets. One can then use the  uniformization obtained above to map $M_T$
onto  the Riemann  sphere, each  triangle being  mapped to  either the
upper or lower  hemisphere according to the order in  which the vertex
types occur  on its  boundary. This construction  leads to  a covering
$f_T:M_T \to \mathbb C^\ast$ which  is a Belyi function, and $(Z,P,O)$
is   then   the  exact   constellation   associated   to  $f_T$;   see
Figure~\ref{fig:random} for an illustration in the case of genus $1$.

This  gives a  justification  to the  predictions at  the  end of  the
previous  section: since  each  exact constellation  gives  rise to  a
triangulation, it can be obtained  (up to automorphism) from that same
triangulation. In turn, since the  number of triangulations of a given
size is finite,  this implies that the number  of exact constellations
of given multiplicities, counted up to automorphism, is itself finite,
and so is the set of values  of $\tau$ such that $\mathbb C / (\mathbb
Z + \tau \mathbb Z)$ supports an exact constellation of a given size.

This  approach   is  very  satisfactory  at   the  theoretical  level;
unfortunately,  it   doesn't  lend   itself  very  well   to  explicit
computations, and performing this program  numerically in a direct way
would be difficult.

\subsection{Number-theoretic statement}
\label{sec:belyi_n}

There is  a deep connection  between the complex  structure introduced
above and the arithmetic properties of the underlying Riemann surface,
having its origin in the following equivalence:

\begin{theorem}[Belyi]
  Let  $C$ be  a non-singular  algebraic  curve: then  there exists  a
  ramified  covering  $C  \to  \mathbb C^\ast$,  ramified  only  above
  $\{0,1,\infty\}$,   if   and   only   if   $C$   is   defined   over
  $\bar {\mathbb Q}$.
\end{theorem}

We will mostly be concerned with the ``only if'' part in what follows;
notice that  when we build  a Riemann surface from  gluing equilateral
triangles, it  automatically comes with a  constellation $(Z,P,O)$ and
such a  covering, namely the  function $f$ constructed  above, sending
the points in $Z$ (resp.\ $P$,  $O$) to $0$ (resp.\ $\infty$, $1$) and
unramified above the other points.

What     it     means     for     us     is     that     the     torus
$\mathbb C / (\mathbb Z +  \tau \mathbb Z)$ built from a triangulation
of  genus   $1$,  seen   as  an  elliptic   curve,  is   defined  over
$\bar   {\mathbb  Q}$.   In   particular,  the   modulus  $\tau$   can
theoretically be  computed explicitly. In  practice, this seems  to be
possible only in very few cases, and in each of them, requires the use
of particular symmetries of the triangulation.

\section{Numerical computation of Belyi maps: genus $0$}
\label{sec:num}

We now turn  to the main point of this  paper, namely a semi-numerical
strategy  to exactly  compute the  branched covering  associated to  a
given hypermap.  As argued above,  this can be reduced  to determining
the exact constellation  associated to the map, which  lends itself to
numerical approaches. We start with the case of hypermaps drawn on the
sphere, which benefits  from lighter notation while  retaining most of
the features of our approach.

Fix       three        tuples       of        positive       integers,
$d^{0}      =      (d_i^{0})_{1      \leq      i      \leq      n_0}$,
$d^{\infty}   =  (d_i^{\infty})_{1   \leq   i   \leq  n_\infty}$   and
$d^{1}  = (d_i^{1})_{1  \leq  i  \leq n_1}$  with  the  same sum  $N$,
satisfying the genus-$0$ condition
$$n_0+n_\infty+n_1  =   N+2,$$  and  let  $\mathfrak   C  =  \mathfrak
C_{d^0,d^\infty,d^1}$
be the space of all  non-degenerate constellations on $\mathbb C^\ast$
with  signature $(d^0,d^\infty,d^1)$  (which are  all balanced  and of
genus $0$). $\mathfrak C$ is a complex manifold of dimension $N+2$. If
$\mathcal C = (Z,P,O) \in \mathfrak C$, let
$$f_{\lambda,  \mathcal C}  : z  \mapsto \lambda  \prod_{(w,d) \in  Z}
(z-w)^d         \prod_{(w,d)        \in         P}        (z-w)^{-d}$$
be defined as above (omitting the star at $\infty$ in the products, if
any),            and            consider            the            map
$\Phi : \mathbb C \times \mathfrak C \to \mathbb C^N$ defined by
$$\Phi(\lambda,   \mathcal   C)    :=   (f(o_1)-1,   f'(o_1),   \ldots,
f^{(d_1^1-1)}(o_1),   f(o_2)-1,  f'(o_2),   \ldots,  f^{(d^1_{n_1}-1)}
(o_{n_1})),$$
where  to lighten  notation we  let $f  = f_{\lambda,\mathcal  C}$ and
$\mathcal C  = (Z,P,O)$  with $O  = ((o_i,d^1_i))$.  The constellation
$\mathcal C$ is exact if and only if there exists $\lambda \in \mathbb
C$ such that $\Phi(\lambda,\mathcal C) = (0,\ldots,0)$.

To normalize the map embeddings, let  $\mathfrak C^\ast$ be the set of
all constellations  $\mathcal C =  (Z,P,O) \in \mathfrak C$  such that
$z_1=0$, $p_1=\infty$ and $o_1=1$. $\mathfrak C^\ast$ is a manifold of
complex  dimension $N-1$  which  can  be seen  as  an  open subset  of
$\mathbb C^{N-1}$ by listing the locations of all the stars except for
$z_1$, $p_1$ and $o_1$. The  restriction $\Phi^\ast : \mathbb C \times
\mathfrak C^\ast \to  \mathbb C^N$ can thus  be seen as a  map from an
open  subset of  $\mathbb  C^N$  to $\mathbb  C^N$,  which is  clearly
analytic in all its variables.

From the previous  discussion, the set of  preimages of $(0,\ldots,0)$
by $\Phi^\ast$ is  finite, and each of its elements  corresponds to an
exact constellation and to the  associated hypermap and covering. From
this, the general structure of the algorithm is rather clear: starting
from a  tripartite hypermap  $M$ and  the associated  triangulation of
genus $0$,  to find its  constellation $\mathcal C_M$ we  will perform
the following.
\begin{itemize}
\item Step 1: find an approximation $\mathcal C_0$ of $\mathcal C_M$;
\item Step 2: refine the  approximation via an iterative scheme, using
  $\mathcal C_0$ as a starting point;
\item  Step 3:  verify that  the solution  obtained is  the right  one
  (topologically); otherwise go back to step~1 with better precision;
\item Step 4: identify the coefficients as algebraic numbers;
\item  Step 5:  verify that  the solution  obtained is  the right  one
  (algebraically).
\end{itemize}

We now describe  each of these steps in some  detail, focusing more on
the first two which constitute the main contribution of the section.

\subsection{Step 1: Approaching the complex structure}
\label{sec:cp}

To  initialize the  algorithm,  we  need to  get  an approximation  of
$\mathcal  C_M$.  This  can  be  done  by  approximating  the  complex
structure derived  from $M$, and  there are several  options available
here.  From  the  complex  analytical   point  of  view,  one  natural
possibility would be to start from  an arbitrary embedding of $M$ into
the sphere with  straight edges, and to identify  the uniformizing map
as  a solution  to a  Beltrami equation  with a  Beltrami differential
taken to be  constant on each of the faces.  This has been implemented
in a few cases but is  quite involved~\cite{cannon:RMT}, so we chose a
simpler way, following~\cite{BK:cp}.

Given a triangulation  $T$ of the sphere, there  exists a \emph{circle
  packing} with the combinatorics of  $T$, namely a collection $(C_v)$
of circles indexed by the vertex  set of $T$, such that the associated
disks  have disjoint  interiors  and  such that  $C_v$  and $C_w$  are
tangent if and only  if $v$ and $w$ are adjacent  vertices of $T$. The
circle packing is unique up to Möbius transformations, and it provides
an embedding of $T$ in which the vertices are mapped to the centers of
the circles and the  edges are each the union of  two circle radii; in
particular,  if  $T$  is  tripartite   it  provides  also  a  balanced
constellation $\mathcal  C_{cp}(T)$. For  general reference  on circle
packings, we refer the reader to~\cite{stephenson:circle}

Our   first   approximation  of   $\mathcal   C_M$   will  simply   be
$\mathcal  C_{cp}(M)$. In  practice, it  turns  out to  often be  good
enough for  our purposes,  but it  is not  always the  case. To  get a
better one, we  follow the strategy of \cite{BK:cp} and  refine $M$ in
the following way. Given a triangulation,  one can replace each of its
faces by $4$  new triangles by inserting a vertex  at the mid-point of
each  of  its  edges  ---  see\  Figure~\ref{fig:sub}.  Iterating  the
procedure   starting  from   $M_0   :=  M$,   one   gets  a   sequence
$(M_n)_{n \geq  0}$ of  triangulation, which  are all  tripartite. The
main result is then the following:

\begin{figure}[h]
  \centering
  \includegraphics{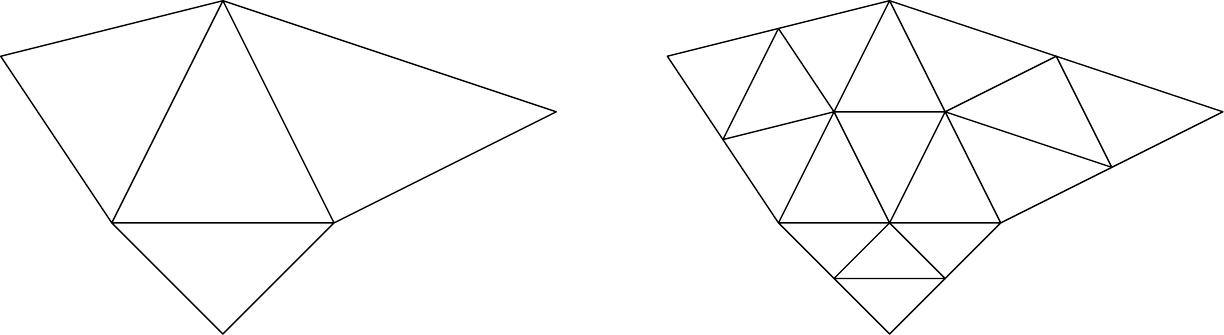}
  \par\bigskip\bigskip
  \includegraphics{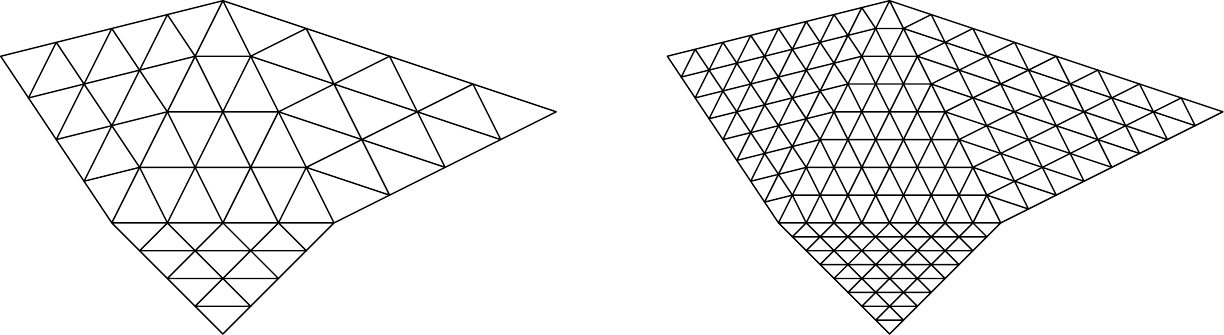}
  \caption{Successive refinements of a triangulation.}
  \label{fig:sub}
\end{figure}

\begin{theorem}[Theorem~4.8 in \cite{BK:cp}]
  \label{thm:sub}
  Under appropriate normalization, $\mathcal C_{cp}(M_n)$ converges to
  $\mathcal C_M$ as $n\to\infty$.
\end{theorem}

This  is  exactly   what  we  needed:  we  now  have   a  sequence  of
constellations  which  approaches $\mathcal  C_M$  and  will serve  as
starting points  for the iterative method  in the next step.  We first
run        steps       2        and       3        starting       from
$\mathcal C_{cp}(M_0) =  \mathcal C_{cp}(M)$, then if  that fails from
$\mathcal  C_{cp}(M_1)$,  $\mathcal  C_{cp}(M_2)$,~\ldots\  until  the
first  one that  succeeds  ---  which is  guaranteed  to happen  after
finitely      many     subdivisions      by     Theorems~\ref{thm:sub}
and~\ref{thm:newton}.

\subsection{Step 2: Newton's method}
\label{sec:newton}

This part is quite standard, but as far as we can tell it has not been
implemented in  this setup before:  one can apply  a multi-dimensional
Newton's method to build a  sequence of constellations which converges
to $\mathcal C_M$, provided one starts close enough to it.

More  specifically, still  identifying  $\mathfrak C^*$  with an  open
subset           of          $\mathbb           C^{N-1}$,          let
$\Psi  :  \mathbb C  \times  \mathfrak  C^\ast  \to \mathbb  C  \times
\mathfrak C^\ast$ be defined by
$$\Psi((\lambda,\mathcal    C))     :=    (\lambda,\mathcal     C)    -
J_\Phi((\lambda,\mathcal    C))^{-1}   \Phi((\lambda,\mathcal    C))$$
whenever     the    Jacobian     $J_\Phi$     is    nonsingular     at
$(\lambda,\mathcal  C)$. The  main statement  of this  section is  the
convergence of Newton's method:

\begin{theorem}
  \label{thm:newton}
  Let $\lambda_M$  be the  canonical normalization of  $\mathcal C_M$:
  there  exists a  neighborhood $U$  of $(\lambda_M,\mathcal  C_M)$ in
  $\mathbb C \times \mathfrak  C^\ast$ satisfying $\Psi(U) \subset U$,
  on which $J_\Phi$ is everywhere nonsingular, and such that, whenever
  $(\lambda,\mathcal C)  \in U$,  the sequence defined  inductively by
  $x_0 = (\lambda, \mathcal C)$ and $x_{n+1} = \Psi(x_n)$ converges to
  $(\lambda_M, \mathcal C_M)$ as $n \to \infty$.
\end{theorem}

It will be more convenient to set  up the computations in the proof in
terms of  logarithmic derivatives, so let  $g_{\lambda,\mathcal C}(z)$
be  a determination  of $\log  (f_{(\lambda,\mathcal C)}(z))$  that is
jointly continuous in $(\lambda,\mathcal C)$  in a neighborhood $V$ of
$(\lambda_M,\mathcal C_M)$ and in $z$  in neighborhoods of each of the
$o_i$, and chosen  such that $g_{(\lambda_M, \mathcal  C_M)}(o_i) = 0$
for  every  $i  \in \{  1,  \ldots,  n_1  \}$.  Let $\Lambda$  be  the
counterpart of $\Phi$  defined in $V$ in terms of  $g$ instead of $f$,
namely:
$$\Lambda(\lambda, \mathcal C) := (g(o_1), g'(o_1), \ldots,
g^{(d_1^1-1)}(o_1),   g(o_2),   g'(o_2),   \ldots,   g^{(d^1_{n_1}-1)}
(o_{n_1})).$$

\begin{lemma}
  There is a constant $C\in\mathbb  C \setminus \{0\}$, depending only
  on  the signature  $(d^0, d^\infty,  d^1)$, such  that the  Jacobian
  determinant of $\Lambda$ at $(\lambda_M, \mathcal C_M)$ is given by
  $$\frac C \lambda \frac
  {\prod_{1  \leq  i <  j  \leq  n_1} (o_j-o_i)^{\tilde  d^1_i  \tilde
      d^1_j}. \prod_{2 \leq i < j \leq n_0} (z_j-z_i). \prod_{2 \leq i
      <    j     \leq    n_\infty}     (f_j-f_i).    \prod_{i=2}^{n_0}
    \prod_{j=2}^{n_\infty}   (f_j-z_i)}    {\prod_{i=1}^{n_1}   \left(
      \prod_{j=2}^{n_0}            (z_j-o_i)^{\tilde           d^1_i}.
      \prod_{j=2}^{n_\infty}   (f_j-o_i)^{\tilde  d^1_i}   \right)},$$
  where   $\tilde  d_i^1   :=   d_i^1-1+\delta_i^1$.  In   particular,
  $\det J_\Lambda((\lambda_M, \mathcal C_M)) \neq 0$.
\end{lemma}

\begin{proof}
  We  first  need  to  compute  the entries  of  the  Jacobian  matrix
  $J_\Lambda((\lambda_M, \mathcal C_M))$.  The iterated derivatives of
  $g$         are          given         for          $k>0$         by
  $$\frac  {(-1)^{k+1}} {(k-1)!}  g^{(k)}(w) =  \sum_{i=1}^{n_0} \frac
  {d^0_i}  {(w-z_i)^k}  -   \sum_{i=1}^{n_\infty}  \frac  {d^\infty_i}
  {(w-z_i)^k}$$
  so     the     partial     derivatives     are     all     explicit:
  $$\partial_\lambda  g^{(k)}(o_i)  =   \frac  1  \lambda  \delta_k^0,
  \qquad     \partial_{o_j}      g^{(k)}(o_i)     =     g^{(k+1)}(o_i)
  \delta_i^j,$$
  $$\partial_{z_j}   g^{(k)}(o_i)  =   \frac  {d_j^0   (-1)^{k+1}  k!}
  {(o_i-z_j)^{k+1}},  \qquad   \partial_{f_j}  g^{(k)}(o_i)   =  \frac
  {d_j^\infty        (-1)^{k+1}        k!}        {(o_i-f_j)^{k+1}}.$$
  In addition,  at the point  $(\lambda_M,\mathcal C_M)$ we  know that
  $g^{(k)}(o_i)=0$ as  soon as $k <  d_i^1$, so up to  terms depending
  only on the  signature (and which are products  of $(-1)^k$, $d_j$-s
  and factorials), the Jacobian determinant is proportional to
  $$\adjustbox{max width=\linewidth}{\(\begin{vmatrix}
    1/\lambda & 0 \cdots & 0 & 1/\lambda & 0 \cdots & 0 &\cdots &
    1/\lambda & 0 \cdots & 0 \\
    (o_1-z_2)^{-1} &  \cdots &  (o_1-z_2)^{-d^1_1} &  (o_2-z_2)^{-1} &
    \cdots & (o_2-z_2)^{-d^1_2} & \cdots &
    (o_{n_1}-z_2)^{-1} & \cdots & (o_{n_1}-z_2)^{-d^1_{n_1}} \\
    \vdots & & \vdots & \vdots & & \vdots & \vdots & \vdots & & \vdots
    \\
    (o_1-z_{n_0})^{-1}   &    \cdots   &    (o_1-z_{n_0})^{-d^1_1}   &
    (o_2-z_{n_0})^{-1} & \cdots & (o_2-z_{n_0})^{-d^1_2} & \cdots &
    (o_{n_1}-z_{n_0})^{-1} & \cdots & (o_{n_1}-z_{n_0})^{-d^1_{n_1}} \\
    (o_1-f_2)^{-1} &  \cdots &  (o_1-f_2)^{-d^1_1} &  (o_2-f_2)^{-1} &
    \cdots & (o_2-f_2)^{-d^1_2} & \cdots &
    (o_{n_1}-f_2)^{-1} & \cdots & (o_{n_1}-f_2)^{-d^1_{n_1}} \\
    \vdots & & \vdots & \vdots & & \vdots & \vdots & \vdots & & \vdots
    \\
    (o_1-f_{n_\infty})^{-1} &  \cdots &  (o_1-f_{n_\infty})^{-d^1_1} &
    (o_2-f_{n_\infty})^{-1}  & \cdots  & (o_2-f_{n_\infty})^{-d^1_2}  &
    \cdots & (o_{n_1}-f_{n_\infty})^{-1} & \cdots
    & (o_{n_1}-f_{n_\infty})^{-d^1_{n_1}} \\
    0 & \ldots & 0 & 0 & \ldots & 1 & \cdots & 0 & \cdots & 0 \\
    \vdots & & \vdots & \vdots & & \vdots & \vdots & \vdots & & \vdots
    \\
    0 & \ldots & 0 & 0 & \ldots & 0 & \cdots & 0 & \cdots & 1
  \end{vmatrix}\)}$$
  where  the columns  are the  components of  $\Lambda$ and  the lines
  correspond  to  partial  derivatives  with  respect,  in  order,  to
  $\lambda$,  the  $(z_i)_{i\geq2}$,   the  $(f_i)_{i\geq2}$  and  the
  $(o_i)_{i\geq2}$ (remember that we  normalized all constellations to
  have $z_1=0$,  $o_1=1$ and  $f_1=\infty$, so they  do not  appear as
  variables here, and  that this makes the  matrix square). Developing
  the  determinant along  its last  $(n_1-1)$ lines  shows that  it is
  equal to
  $$\adjustbox{max width=\linewidth}{\(\begin{vmatrix}
    1/\lambda & 0 \cdots & 0 & 1/\lambda & 0 \cdots & 0 &\cdots &
    1/\lambda & 0 \cdots & 0 \\
    (o_1-z_2)^{-1} &  \cdots &  (o_1-z_2)^{-d^1_1} &  (o_2-z_2)^{-1} &
    \cdots & (o_2-z_2)^{1-d^1_2} & \cdots &
    (o_{n_1}-z_2)^{-1} & \cdots & (o_{n_1}-z_2)^{1-d^1_{n_1}} \\
    \vdots & & \vdots & \vdots & & \vdots & \vdots & \vdots & & \vdots
    \\
    (o_1-z_{n_0})^{-1}   &    \cdots   &    (o_1-z_{n_0})^{-d^1_1}   &
    (o_2-z_{n_0})^{-1} & \cdots & (o_2-z_{n_0})^{1-d^1_2} & \cdots &
    (o_{n_1}-z_{n_0})^{-1} & \cdots & (o_{n_1}-z_{n_0})^{1-d^1_{n_1}} \\
    (o_1-f_2)^{-1} &  \cdots &  (o_1-f_2)^{-d^1_1} &  (o_2-f_2)^{-1} &
    \cdots & (o_2-f_2)^{1-d^1_2} & \cdots &
    (o_{n_1}-f_2)^{-1} & \cdots & (o_{n_1}-f_2)^{1-d^1_{n_1}} \\
    \vdots & & \vdots & \vdots & & \vdots & \vdots & \vdots & & \vdots
    \\
    (o_1-f_{n_\infty})^{-1} &  \cdots &  (o_1-f_{n_\infty})^{-d^1_1} &
    (o_2-f_{n_\infty})^{-1} & \cdots & (o_2-f_{n_\infty})^{1-d^1_2} &
    \cdots     &     (o_{n_1}-f_{n_\infty})^{-1}    &     \cdots     &
    (o_{n_1}-f_{n_\infty})^{1-d^1_{n_1}}
  \end{vmatrix}\)}$$
  which can  now be identified  as a  rational fraction in  the $z_i$,
  $f_i$ and $o_i$.

  The denominator is easier to compute  first: it just consists in the
  product  of  the  terms  $(o_i-z_j)^k$ and  $(o_i-f_j)^k$  with  the
  highest    power    appearing    in    the    determinant,    namely
  $d_i^1-1+\delta_i^1  = \tilde  d_i^1$  (which is  the  width of  the
  matrix block corresponding to $o_i$). Multiplying each line by those
  factors which appear in it, and the first one by $\lambda$, one gets
  a matrix $A$ with polynomial entries  in the $z_i$, $f_i$ and $o_i$;
  let $P$  be the determinant  of that matrix.  To compute it,  we now
  forget about the origin of the matrix and note that if any two among
  $\{z_i\} \cup \{f_i\}$  are equal, then two of the  lines of $A$ are
  equal, hence $P$ is divisible by  the product of all the differences
  between these values.

  If on the other hand $o_i=o_j$ for  some $i \neq j$, then two blocks
  of vertical lines become equal, so $P$ is divisible by some power of
  $(o_i-o_j)$. Getting  the value of  the power  is a little  bit more
  involved  but again  can be  done  very explicitly.  Start with  the
  following identity (where $a \neq x$ and $k\in\mathbb Z_+$):
  \begin{equation}
    \label{eq:expdet}
    \frac  1 {a-x}  = \frac  1 a  +  \frac x  {a^2} +  \cdots +  \frac
    {x^{k-1}} {a^{k}} + \frac {x^{k}} {a^{k} (a-x)}.
  \end{equation}
  This can be applied  to get an expansion of the  first column in the
  block  corresponding   to  $o_j$:  for  the   second  line,  letting
  $a=(o_i-z_2)$,  $x=(o_i-o_j)$  and $k=d_i^1-1+\delta_i^1$  gives  an
  expansion of  $(o_j-z_2)^{-1}$ in terms  of the entries in  the same
  line  in the  block of  $o_i$, with  coefficients depending  only on
  $(o_i-o_j)$  and hence  being  the  same across  all  lines, plus  a
  remainder where $(o_i-o_j)^k$ is in factor: more explicitly,
  \begin{equation}
    \frac  1  {o_j-z_2}  =   \frac  1  {o_i-z_2}  +  \frac
    {o_i-o_j} {(o_i-z_2)^2} + \cdots + \frac
    {(o_i-o_j)^{k-1}}   {(o_i-z_2)^{k}}    +   \frac   {(o_i-o_j)^{k}}
    {(o_i-z_2)^{k} (o_j-z_2)}.
  \end{equation}

  The  next columns  can  be  expanded in  a  similar  way, where  the
  expansion of  the $\ell$-th column  in the  block of $o_j$  uses the
  last $(d_i^1-1+\delta_i^1)-(\ell-1)$  column of  the block  of $o_i$
  and the remainders  in the first $(\ell-1)$ columns of  the block of
  $o_j$,  with  coefficients  depending  only  on  $(o_i-o_j)$  and  a
  remainder term where the same power of $(o_i-o_j)$ as before factors
  out.  These  expansions  are obtained  from  successive  derivatives
  of~\eqref{eq:expdet} with  respect to $x$;  the first one  being for
  instance
  $$\frac 1 {(a-x)^2} =  \frac 1 {a^2}  + \frac {2x} {a^3}
  \cdots +  \frac {(k-1)x^{k-2}}  {a^{k}} +  \frac {k  x^{k-1}} {a^{k}
    (a-x)} + \frac {x^k} {a^k (a-x)^2}.$$
  Overall,  the power  of  $(o_i-o_j)$  in $P$  obtained  this way  is
  exactly the product of the widths of the corresponding blocks in the
  matrix.

  To summarize, $P$ is divisible by  the numerator in the statement of
  the  theorem. Matching  the  degrees shows  that  the ratio  between
  $\det J_\Lambda$ and the formula in  the statement of the lemma is a
  constant, which is the main claim.  The fact that the determinant is
  not  zero   is  then   a  direct  consequence   of  the   fact  that
  $\mathcal C_M$ is non-degenerate.
\end{proof}

\begin{proof}[Proof of Theorem~\ref{thm:newton}]
  First, we can  replace $f$ with $e^g$ in the  expression for $\Phi$,
  and compute the  partial derivatives appearing in  $J_\Phi$ in terms
  of  those appearing  in $L_\Lambda$.  By repeated  use of  the chain
  rule, commutativity of derivatives and  the fact that $\mathcal C_M$
  is exact, we get for example
  $$\partial_{z_i} [(f_{\lambda, \mathcal C})^{(j)}(o_k)]
  = (f_{\lambda,  \mathcal C}  \partial_{z_i} g_{\lambda,  \mathcal C}
  )^{(j)}(o_k) =  \sum_{\ell=0}^j \binom j \ell  (f_{\lambda, \mathcal
    C}^{(\ell)}  \partial_{z_i}   g_{\lambda,  \mathcal  C}^{(j-\ell)}
  )(o_k) =  \partial_{z_i} [(g_{\lambda, \mathcal C}  )^{(j)}(o_k)] $$
  at the  point $(\lambda_M, \mathcal  C_M)$ as  soon as $j  < d^1_k$.
  Similar  computations  in the  other  variables  show that  in  fact
  $J_\Lambda$    and    $J_\Phi$    are    equal    at    the    point
  $(\lambda_M,     \mathcal      C_M)$     and      in     particular,
  $J_\Phi((\lambda_M, \mathcal C_M))$ is non-singular.

  The remainder of  the proof is a completely  standard application of
  the usual  Newton algorithm:  it suffices to  use the  smoothness of
  $\Phi$  in all  variables to  show that  $J_\Phi((\lambda,\Phi))$ is
  non-singular in  a neighborhood of $(\lambda_M,  \mathcal C_M)$, and
  to expand $\Psi$ at $(\lambda_M, \mathcal C_M)$ to obtain a bound of
  the                                                             form
  $$\| \Psi((\lambda, \mathcal C)) - (\lambda_M, \mathcal C_M) \| \leq
  C \|  (\lambda, \mathcal C)  - (\lambda_M, \mathcal C_M)  \|^2$$ for
  $(\lambda, \mathcal C)$ close enough to $(\lambda_M, \mathcal C_M)$.
\end{proof}

To summarize our  construction so far, we are essentially  done at the
theoretical  level:  there is  a  neighborhood  $U$  of our  point  of
interest $(\lambda_M, \mathcal C_M)$  from which Newton's algorithm is
guaranteed to  converge, and  iterating the  subdivision in  the first
step  will bring  the  approximation within  $U$  after finitely  many
steps.  In  addition, since  the  convergence  is quadratic,  one  can
iterate  a  computer implementation  until  two  successive values  are
indistinguishable  within   machine  precision  (or   chosen  extended
precision) at  very little cost, meaning  that we can get  a numerical
approximation of  $(\lambda_M, \mathcal C_M)$ with  arbitrarily chosen
precision in a reasonable computing time.

\bigskip

A significant issue  in practice though is that none  of the two steps
is quantitative. The speed of convergence of circle packing embeddings
to  the  uniformizing  map  is   not  well  understood  (although  our
simulations  as  well   as  the  numerical  experiments   at  the  end
of~\cite{BK:cp} suggest that  it should be polynomial  in the diameter
of the smallest circle and  exponential in the number of refinements),
and the neighborhood  $U$ we would be able to  explicitly construct by
keeping  track of  all constants  implicit  in the  proof above  would
certainly  be   much  smaller   than  the   basin  of   attraction  of
$(\lambda_M, \mathcal C_M)$.

What this means  is that the number of subdivisions  we should perform
in  step~1 to  be certain  to have  convergence would  make that  step
computationally unfeasible. We now turn to a way around this issue.

\subsection{Step 3: Topological verification}
\label{sec:verif}

As it  turns out,  in practice  the domain  of attraction  of Newton's
method in the  case we are interested  in seems to be  quite large, at
least much larger than continuity  arguments for $\Phi$ would predict.
A  natural procedure  is therefore  to start  Newton's algorithm  from
successive iterations of step~1, starting in fact with no iteration at
all,  and in  each  case to  see  if the  iteration  converges or  not
(quadratic  convergence when  the iteration  is successful  means that
recognizing convergence is very quick).

If  the  iteration does  not  converge,  or  converges to  a  singular
constellation, then we simply refine once more. If it does converge to
a  non-singular pair  $(\lambda_\infty, \mathcal  C_\infty)$, then  we
found a  constellation with the  right signature. It remains  to check
whether it actually is the one we were looking for (and subdivide once
more if it is not).

One can  simply do  it visually,  from a  picture of  the sign  of the
imaginary part of the function $f_{\lambda_\infty, \mathcal C_\infty}$
like those in the figures  of this paper, recovering the triangulation
and just  checking that it  is isomorphic to  $M$. The same  thing can
probably  be  automated,  although  it  is  not  at  all  clear  which
algorithms could  be proved to  work; constructing a  region adjacency
graph  from  the  picture  (thus  recovering the  dual  graph  of  the
triangulation) would be an option, but we did not attempt to implement
it.

In  practice, except  on specially  tuned cases  designed to  test the
numerical stability of  the whole method (similar for  instance to the
map  shown  in  Figure~\ref{fig:star}  with  vertices  of  very  large
degree), the limit is almost always the right one --- though of course
this  statement is  not  of a  mathematical nature  and  is merely  an
empirical observation.

\subsection{Step 4: Lattice reduction}
\label{sec:lll}

At   this   point,   we   obtained  a   numerical   approximation   of
$(\lambda_M, \mathcal  C_M)$ with the  normalization chosen in  such a
way that the function $f_{\lambda_M, \mathcal C_M}$ has a zero at $0$,
a pole at $\infty$ and takes the  value one at $1$. We know in advance
that with such normalization, the  locations of the other zeros, poles
and  ones are  all algebraic  numbers: indeed,  there is  a choice  of
normalization such that all locations  are algebraic and mapping it to
our preferred one  can in turn be  done by applying a  Möbius map with
algebraic coefficients.

This means  that for each of  those algebraic numbers, we  are able to
obtain an approximation to any precision that we want, in a reasonable
computing  time. Identifying  the minimal  polynomial of  an algebraic
number given  such an  approximation is  a much-studied  question, and
there exist a variety  of classical \emph{integer relation algorithms}
to do  it, based on  lattice-reduction methods. Going into  a detailed
description of  such algorithms  is besides the  point of  this paper;
several  implementations   are  freely  available,  and   the  results
presented below  were obtained  using one  of them  (specifically FPLLL,
see~\cite{fplll}).

\bigskip

While the  empirical observation is  that these methods work,  again a
comment of  a more  theoretical nature  is in  order: even  though the
numbers that  we are  interested in  are all  algebraic, the  bounds on
their degree that one can derive from the proof of Belyi's theorem are
enormous and the precision approximation  that is required for lattice
reduction to provably find the right solution is therefore enormous as
well, to the point that making  those bounds quantitative is likely to
be of little practical use.

A more  interesting question is  that of the choice  of normalization.
The one we chose was convenient from the implementation point of view,
as it made the formulas in Step~2 explicit, but there is no reason why
it would lead to the algebraic  numbers of the lowest possible degree,
thus  compounding the  previous remark.  Other choices  are of  course
available: for  instance one might  want the sum with  multiplicity of
all zeros to  be $0$ (to get a vanishing  coefficient in the numerator
of $f_{\lambda_M,  \mathcal C_M}$).  More convincingly,  if $M$  has a
non-trivial automorphism group, so does its constellation, and one may
want  (some   of)  the   corresponding  automorphisms  to   be  affine
transformations of the plane.

This  however has  little impact  on  the implementation,  as one  can
always go from one normalization  to another after having obtained the
approximation in Step~3.

\subsection{Step 5: Algebraic verification}
\label{sec:algver}

The last  validation step is of  a purely algebraic nature:  given the
list of  locations as  roots of integer  polynomials, verify  that the
constellation  they  form is  indeed  exact,  in  which case  we  have
achieved our programme  of computing the Belyi  function associated to
the  map $M$  explicitly. This  can be  rewritten as  a collection  of
algebraic  equations  that  they  must satisfy,  and  can  be  checked
explicitly either by hand (for smaller cases) or by a computer algebra
system; again, if the validation failed, it means that the polynomials
obtained in  the previous step are  erroneous, and one can  re-run the
lattice  reduction  from  higher precision  approximations,  with  the
guarantee that after finitely many round trips the right solution will
be found.

\begin{remark}
  This last  step is usually  presented as  the starting point  of the
  computation of  Belyi functions:  namely, starting from  a hypermap,
  obtain  a system  of polynomial  equations  in the  location of  the
  ramification points,  and then use  elimination theory to  solve the
  system (usually using  Gröbner bases). This works  well in practice,
  but does not  seem to extend well to higher  genus beyond very small
  maps.
\end{remark}

\section{Numerical computation of Belyi maps: genus $1$}
\label{sec:torus}

We  now turn  to  the case  of  triangulations of  the  torus, and  to
elliptic Belyi functions. As mentioned  earlier, the main structure of
the  construction is  extremely similar  to the  one described  in the
previous section,  and we  will focus on  the differences  rather than
giving   a  complete   description,  using   the  same   notation  for
corresponding but slightly different objects  where it doesn't lead to
confusion.

\begin{figure}
  \centering
  \includegraphics[width=.8\linewidth]{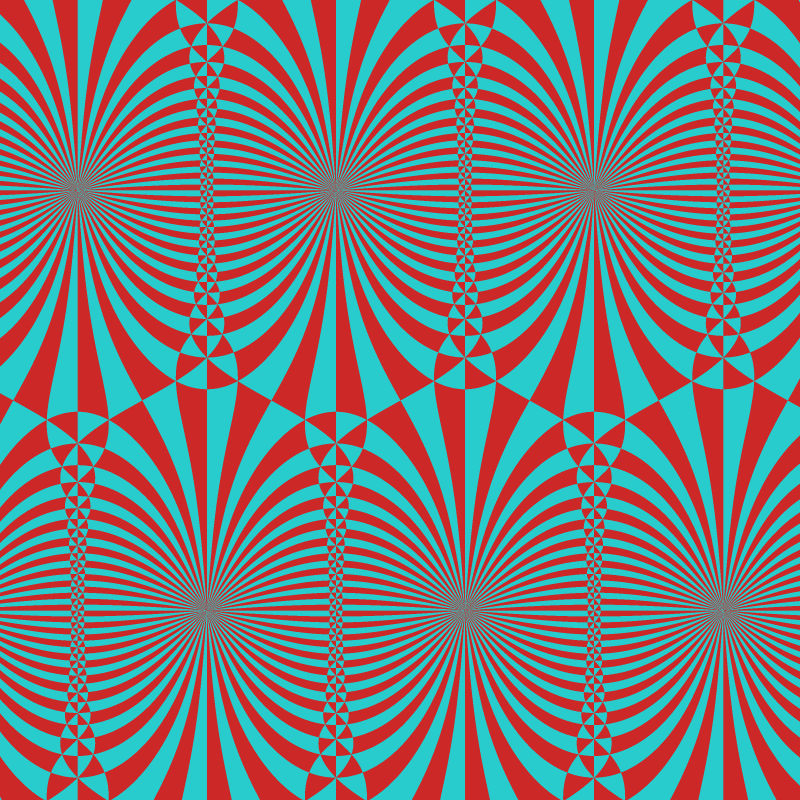}
  \caption{Uniformization  of  a triangulation  of  genus  1 with  one
    vertex  of high  degree  (which can  be used  as  a test-case  for
    numerical stability of the algorithm).}
  \label{fig:star}
\end{figure}

We  will work  on the  space of  all non-degenerate  constellations of
genus  $1$ and  of  a  given signature  defined  on  a complex  torus;
however, there is an invariant (the  modulus of the torus) which we do
not know in advance. This means that the objects we will work with are
in fact triples of the form $(\tau, \lambda, \mathcal C)$ where $\tau$
is  the modulus  of a  torus, $\lambda$  is a  normalizing factor  and
$\mathcal     C$    is     a     constellation     in    the     torus
$\mathbb  T_\tau =  \mathbb C  / (\mathbb  Z +  \tau \mathbb  Z)$, and
having such  a representation of the  Riemann surface on which  we are
working will be very convenient.

So    again,    fix    three     tuples    of    positive    integers,
$d^{0}      =      (d_i^{0})_{1      \leq      i      \leq      n_0}$,
$d^{\infty}   =  (d_i^{\infty})_{1   \leq   i   \leq  n_\infty}$   and
$d^{1}  = (d_i^{1})_{1  \leq  i  \leq n_1}$  with  the  same sum  $N$,
satisfying the genus-$1$ condition
$$n_0+n_\infty+n_1  =   N,$$  and  let  $\mathfrak   T  =  \mathfrak
T_{d^0,d^\infty,d^1}$
be  the  space  of   all  triples  $(\tau,\lambda,\mathcal  C)$  where
$\tau  \in  \mathbb  H$,  $\lambda\in\mathbb C  \setminus  \{0\}$  and
$\mathcal C$  is a  non-degenerate constellations on  $\mathbb T_\tau$
with  signature  $(d^0,d^\infty,d^1)$.  $\mathfrak  T$  is  a  complex
manifold of  dimension $N+2$. We  will always  see $\mathbb T$  as the
quotient     of    the     complex     plane     by    the     lattice
$\mathbb Z + \tau \mathbb Z$,  represent the locations of the stars of
$\mathcal C$ as complex numbers and  identify $\mathfrak T$ as an open
subset of  $\mathbb C^{N+2}$; if  $\mathcal C = (Z,P,O)  \in \mathfrak
C$, let
$$f_{\tau,  \lambda, \mathcal  C}  : z  \mapsto  \lambda \frac  {\prod
  \zeta(z-z_i)^{d_i^0}}       {\prod      \zeta(z-p_i)^{d_i^\infty}}$$
be     defined     as     above,      and     consider     the     map
$\Phi : \mathfrak T \to \mathbb C^N$ defined by
$$\Phi((\tau, \lambda, \mathcal C)) := (f(o_1)-1, f'(o_1), \ldots,
f^{(d_1^1-1)}(o_1),   f(o_2)-1,  f'(o_2),   \ldots,  f^{(d^1_{n_1}-1)}
(o_{n_1})),$$
where to  lighten notation we  let $f =  f_{\tau,\lambda,\mathcal C}$.

To normalize the map embeddings and ensure that $f$ is doubly periodic
with periods $1$ and $\tau$ (or in  other words, to have it defined on
the torus $\mathbb T_\tau$), let $\mathfrak  T^\ast$ be the set of all
triples  $(\tau,\lambda,\mathcal   C)  \in  \mathfrak  T$   such  that
\begin{equation}
  \label{eq:zp0}
  \sum d_i^0 z_i = \sum d_i^\infty  p_i = 0.
\end{equation}
$\mathfrak T^\ast$ is a manifold of complex dimension $N$ which can be
seen as an open subset of  $\mathbb C^{N}$ by listing the locations of
all  the   stars  except   for  $z_1$   and  $p_1$;   the  restriction
$\Phi^\ast : \mathfrak  T^\ast \to \mathbb C^N$ can thus  be seen as a
map from  an open subset of  $\mathbb C^N$ to $\mathbb  C^N$, which is
clearly analytic  in all its  variables. We are interested  in finding
one specific  preimage of  $(0,\ldots,0)$ by  $\Phi$, out  of finitely
many.

\bigskip

We  are  now  exactly  in  the same  situation  as  before,  the  only
difference being the definition of $f$,  so we will briefly review the
relevant changes needed to the previous steps.

\subsection{Step 1: Approaching the complex structure}
\label{sec:cp1}

Here the  construction starts from  a triangulation $T$ of  genus $1$,
and  again  there exists  a  unique  $\tau_T  \in  \mathbb H$  and  an
essentially unique circle  packing on the torus  of modulus ${\tau_T}$
having  the combinatorics  of $T$.  Equivalently, the  universal cover
$\hat T$ of  $T$ is a doubly periodic triangulation  of the plane, and
there is  an essentially unique  locally finite circle packing  of the
plane  with the  combinatorics of  $\hat  T$; this  circle packing  is
automatically doubly  periodic, and  can be normalized  in such  a way
that it has the two periods $\{1,\tau_T\}$.

We want to  define the starting point for the  Newton iterative scheme
using $\tau_T$ as the modulus and  the locations of the centers of the
circles as  those of  the stars;  and here  as well,  under successive
refinements, these converge to the constellation corresponding to $M$.
One slight  problem is that there  is no reason why  the normalization
condition~\eqref{eq:zp0} can be  made to hold by a  suitable choice of
embedding of  a circle  packing ---  but we can  at least  ensure that
$\sum  d_i^0 z_i=0$,  which  makes the  embedding  unique, and  simply
ignore the sum of the $d_i^\infty p_i$  (which will tend to $0$ as the
triangulation is refined more and more).

\subsection{Step 2: Newton's method}
\label{sec:newton1}

This   part   works   exactly   the   same   way   as   before,   with
$$\Psi((\tau, \lambda, \mathcal  C)) := (\tau, \lambda,  \mathcal C) -
J_{\Phi}  ((\tau,  \lambda,  \mathcal  C))^{-1}  \Phi((\tau,  \lambda,
\mathcal                                                        C)),$$
and the  only point to  check is  the non-singularity of  the Jacobian
matrix at the  point $(\tau_M, \lambda_M, \mathcal C_M)$.  This can be
done in a very  similar way as on the sphere, as  far as the variables
besides  $\tau$   are  concerned:   for  fixed  $\tau$   the  Jacobian
determinant is an elliptic function in all its other variables, and it
factorizes for the same reason as in the rational case. Alternatively,
one can also get the non-singularity of the Jacobian by moving all the
locations  of  the points  closer  and  closer  to the  origin  (while
satisfying  the  constraints  listed  above), and  noticing  that  the
asymptotic behavior  of the determinant, once  properly normalized, is
given by the case of genus $0$.

The variable $\tau$ is a bit more problematic, because the derivatives
in $\tau$ of  the functions involved in $\Phi$ are  not as explicit as
the derivatives in $z$. Instead, one way  to go around the issue is to
argue that a different way of  proceeding, from the data of a toroidal
triangulation, is  as follows:  take a  large square  of $N  \times N$
periods in the  universal cover of $M$,  and stitch it with  a copy of
itself along the  boundary of the square to obtain  a triangulation of
the sphere. This triangulation can  be uniformized (for instance using
the method  in the  previous section)  and the  local behavior  of the
uniformizing function near the center  of the square, as $N\to\infty$,
becomes periodic and asymptotic to the uniformizing map of $M$ itself.
In particular, the  fact that the Jacobian is  non-degenerate in genus
$1$ can be extracted from the corresponding statement in genus $0$.

\begin{remark}
  While that  last remark gives  an alternative strategy to  solve our
  initial problem  as well, it  has two  main drawbacks: first,  it is
  difficult  to estimate  the level  of  precision that  one would  be
  obtained as a function of $N$  (presumably it would be polynomial in
  $N$);  second,  and  very  related,  is that  to  get  the  kind  of
  approximation needed to apply the next  step, the value of $N$ would
  have to be  taken so large that  there would be no  hope of actually
  implementing  the programme.  It  would still  be  a possibility  to
  obtain the  starting point of  step~$2$ in this way,  thus replacing
  step $1$; but we did not try this route.
\end{remark}

\subsection{Steps 3, 4 and 5: Identification and validation}
\label{sec:verif1}

Here, not  much needs to  be changed at  all except for  notation; the
questions  raised, whether  the successive  objects that  we construct
correspond to the one we are looking for, are the same, and the method
is exactly parallel to that used in genus $0$.

One point needs to  be made though. In the sphere,  the freedom in the
choice of  embedding meant that  finding the constellation  leading to
algebraic numbers  with minimal  degree was an  issue. Here,  there is
much more rigidity from the conditions~\eqref{eq:zp0}, which itself is
very natural, and there is no choice at all, beyond the usual discrete
$SL_2(\mathbb Z)$ action,  in the modulus $\tau$;  indeed the elliptic
curve defined  by $M$ is  uniquely defined. This is  quite convenient,
especially since determining  $\tau$ was our primary goal  in the case
of genus $1$.

\section{A few examples}
\label{sec:ex}

The  front  page  of  this  paper shows  the  Belyi  function  of  the
tripartite refinement  of a uniformly sampled  random triangulation of
the sphere  with $15$ vertices.  We validated  the algorithm of  a few
known cases, such as those  listed in~\cite{AAD:catalog}, but chose to
focus on genus $1$ for the examples of this section.

We list here the Belyi  functions obtained from all the triangulations
of the torus with up to $3$ vertices and all degrees at least equal to
$3$. In  each case, the  triangulation is refined into  its tripartite
refinement  and we  list the  hypermap description  of the  original
triangulation, a  graphical representation  of its  combinatorics, and
then the sign of the imaginary part of the Belyi function and in a few
instances  the exact  value of  the associated  $j$-invariant (or  its
minimal polynomial  in case that  is more useful).  For triangulations
with $4$ vertices we  give one example in the same  form, and only the
graph of the covering map for the others.

\bigskip

In several cases  the original triangulation is  itself tripartite, so
the  representation is  not  minimal;  but the  value  of  $j$ is  not
affected by the  refinement, and the dessin associated  to the initial
triangulation is a subset of that of the refined one.

Each finite triangulation $T$ has countably many coverings $(T_{k,l})$
that  are  themselves  finite triangulations,  having  as  fundamental
domains unions  of finitely many  copies of the fundamental  domain of
$T$ arranged  as a $k  \times l$ rectangle.  Those are omitted  in the
catalog below, and only the minimal one is listed.

\subsection{Size $1$}
\label{sec:exs1}

\newcommand{\exone}[6]{\par\medskip\noindent
  \begin{minipage}{.48\linewidth}\centering
    \adjustbox{max width=\linewidth}{$\sigma = #3$}

    \adjustbox{max width=\linewidth}{$\alpha = #4$}

    \adjustbox{max width=\linewidth}{$\varphi = #5$}

    \bigskip\maxsizebox{\linewidth}{.7\linewidth}{\includegraphics{fig/ex1/s#1/i#2_map.pdf}}

    \vspace{-3ex}$#6$
  \end{minipage}
  \hfill\raisebox{-.24\linewidth}{\includegraphics[width=.48\linewidth]{fig/ex1/s#1/i#2_pic.png}}
  \par\medskip
}

There is only one triangulation of  the torus with one vertex, and its
natural  embedding  is  the   usual  triangular  lattice  composed  of
equilateral triangles:

\exone{1}{1}
{(0\,5\,2\,1\,4\,3)}
{(0\,1)\,(2\,3)\,(4\,5)}
{(0\,2\,4)\,(1\,3\,5)}
{$j=0$}

\subsection{Size $2$}
\label{sec:exs2}

In addition to the double coverings of the previous one, there are two
minimal  triangulations of  the  torus  with two  vertices:  one is  a
refinement  of the  triangular lattice  (with the  same $j$  invariant
equal to $0$) and the other is the face-centered square lattice.

\exone{2}{3}
{(0\,7\,8\,5\,2\,1\,9\,10\,4)\,(3\,11\,6)}
{(0\,1)\,(2\,3)\,(4\,5)\,(6\,7)\,(8\,9)\,(10\,11)}
{(0\,2\,6)\,(1\,4\,8)\,(3\,5\,10)\,(7\,11\,9)}
{$j=0$}

\exone{2}{1}
{(0\,7\,10\,4)\,(1\,9\,6\,3\,11\,8\,5\,2)}
{(0\,1)\,(2\,3)\,(4\,5)\,(6\,7)\,(8\,9)\,(10\,11)}
{(0\,2\,6)\,(1\,4\,8)\,(3\,5\,10)\,(7\,9\,11)}
{$j=1\,728$}

\subsection{Size $3$}
\label{sec:exs3}

There  are $9$  triangulations  of  the torus  with  $3$ vertices  and
minimal  degree at  least $3$  that  are not  coverings of  previously
displayed cases. They turn out to all have rational $j$-invariants.

\exone{3}{6}
{(0\,7\,17\,14\,9\,5\,12\,6\,3\,15\,10\,4)\,(1\,8\,2)\,(11\,16\,13)}
{(0\,1)\,(2\,3)\,(4\,5)\,(6\,7)\,(8\,9)\,(10\,11)\,(12\,13)\,(14\,15)\,(16\,17)}
{(0\,2\,6)\,(1\,4\,9)\,(3\,8\,14)\,(5\,10\,13)\,(7\,12\,16)\,(11\,15\,17)}
{$j=0$}

\exone{3}{1}
{(0\,7\,16\,10\,4)\,(1\,9\,12\,6\,3\,11\,14\,8\,5\,2)\,(13\,15\,17)}
{(0\,1)\,(2\,3)\,(4\,5)\,(6\,7)\,(8\,9)\,(10\,11)\,(12\,13)\,(14\,15)\,(16\,17)}
{(0\,2\,6)\,(1\,4\,8)\,(3\,5\,10)\,(7\,12\,17)\,(9\,14\,13)\,(11\,16\,15)}
{$j=1875$}

\exone{3}{4}
{(0\,7\,14\,10\,4)\,(1\,9\,17\,15\,13\,8\,5\,2)\,(3\,11\,16\,12\,6)}
{(0\,1)\,(2\,3)\,(4\,5)\,(6\,7)\,(8\,9)\,(10\,11)\,(12\,13)\,(14\,15)\,(16\,17)}
{(0\,2\,6)\,(1\,4\,8)\,(3\,5\,10)\,(7\,12\,15)\,(9\,13\,16)\,(11\,14\,17)}
{$j=-3\,072$}

\exone{3}{5}
{(0\,5\,13\,10\,6\,2\,1\,7\,14\,8\,4\,3)\,(9\,16\,12)\,(11\,17\,15)}
{(0\,1)\,(2\,3)\,(4\,5)\,(6\,7)\,(8\,9)\,(10\,11)\,(12\,13)\,(14\,15)\,(16\,17)}
{(0\,2\,4)\,(1\,3\,6)\,(5\,8\,12)\,(7\,10\,15)\,(9\,14\,17)\,(11\,13\,16)}
{$j=-3\,072$}

\exone{3}{3}
{(0\,7\,17\,15\,10\,4)\,(1\,9\,16\,13\,14\,8\,5\,2)\,(3\,11\,12\,6)}
{(0\,1)\,(2\,3)\,(4\,5)\,(6\,7)\,(8\,9)\,(10\,11)\,(12\,13)\,(14\,15)\,(16\,17)}
{(0\,2\,6)\,(1\,4\,8)\,(3\,5\,10)\,(7\,12\,16)\,(9\,14\,17)\,(11\,15\,13)}
{$j = \frac{35\,152}{9}$}

\exone{3}{9}
{(0\,5\,13\,14\,8\,4\,3)\,(1\,7\,17\,15\,10\,6\,2)\,(9\,16\,11\,12)}
{(0\,1)\,(2\,3)\,(4\,5)\,(6\,7)\,(8\,9)\,(10\,11)\,(12\,13)\,(14\,15)\,(16\,17)}
{(0\,2\,4)\,(1\,3\,6)\,(5\,8\,12)\,(7\,10\,16)\,(9\,14\,17)\,(11\,15\,13)}
{$j=-\frac{33\,268\,701}{256}$}

\exone{3}{10}
{(0\,7\,16\,8\,2\,1\,11\,14\,9\,4)\,(3\,15\,12\,6)\,(5\,17\,13\,10)}
{(0\,1)\,(2\,3)\,(4\,5)\,(6\,7)\,(8\,9)\,(10\,11)\,(12\,13)\,(14\,15)\,(16\,17)}
{(0\,2\,6)\,(1\,4\,10)\,(3\,8\,14)\,(5\,9\,16)\,(7\,12\,17)\,(11\,13\,15)}
{$j=\frac{8\,429\,568}{15\,625}$}

\exone{3}{2}
{(0\,7\,14\,8\,5\,2\,1\,9\,16\,10\,4)\,(3\,11\,12\,6)\,(13\,17\,15)}
{(0\,1)\,(2\,3)\,(4\,5)\,(6\,7)\,(8\,9)\,(10\,11)\,(12\,13)\,(14\,15)\,(16\,17)}
{(0\,2\,6)\,(1\,4\,8)\,(3\,5\,10)\,(7\,12\,15)\,(9\,14\,17)\,(11\,16\,13)}
{$j = - \frac{1\,636\,015\,539}{41\,229\,056}$}

\exone{3}{8}
{(0\,5\,13\,14\,11\,8\,4\,3)\,(1\,7\,15\,16\,10\,6\,2)\,(9\,17\,12)}
{(0\,1)\,(2\,3)\,(4\,5)\,(6\,7)\,(8\,9)\,(10\,11)\,(12\,13)\,(14\,15)\,(16\,17)}
{(0\,2\,4)\,(1\,3\,6)\,(5\,8\,12)\,(7\,10\,14)\,(9\,11\,16)\,(13\,17\,15)}
{$j=\frac{116\,634\,423\,954\,432}{1\,977\,326\,743}$}

\subsection{Size $4$}
\label{sec:exs4}

With $4$  vertices there are  already too many triangulations  to make
listing them  very useful. We  still display  one in detail,  and will
only show the picture for the others:

\exone{4}{59}
{(0\,7\,19\,22\,14\,9\,4)\,(1\,11\,21\,23\,16\,8\,2)\,(3\,15\,20\,12\,6)\,(5\,17\,18\,13\,10)}
{(0\,1)\,(2\,3)\,(4\,5)\,(6\,7)\,(8\,9)\,(10\,11)\,(12\,13)\,(14\,15)\,(16\,17)\,(18\,19)\,(20\,21)\,(22\,23)}
{(0\,2\,6)\,(1\,4\,10)\,(3\,8\,14)\,(5\,9\,16)\,(7\,12\,18)\,(11\,13\,20)\,(15\,22\,21)\,(17\,23\,19)}
{$j^2 - 914\,416 \, j + 590\,816\,592 = 0$}

This    is    the    triangulation    used   as    an    example    in
\cite{BI:triangulations,SV:curves},  where   the  elliptic   curve  is
derived formally. It  is defined over $\mathbb Q[\sqrt7]$  and one can
check that the  $j$-invariant obtained by our method is  the right one
(the discriminant of the polynomial  above is $7 \cdot (345\,128)^2$).
Note though that the proof  as detailed in \cite{BI:triangulations} is
4  pages long  and moreover  relies  very strongly  on the  additional
symmetries  of the  triangulation ---  in  the picture  above one  can
readily see  that the embedding  is symmetric under reflection  by the
line going through  $0$ and $1+\tau$ and to another  one orthogonal to
it,  which  both correspond  to  automorphisms  of  order $2$  of  the
triangulation.

This   is  a   strong  indication   that  very   small  examples   are
computationally  difficult to  address, and  slightly larger  ones, or
even cases of small size but no symmetry, are beyond these methods. In
comparison, the programme described  here gets the exact constellation
to machine precision (\emph{i.e.}, to  within $10^{-13}$) in less than
a tenth of a second and  producing enough digits to obtain the minimal
polynomial for  $j$ takes of the  order of $20$ seconds  on a standard
laptop.

\bigskip

\newcommand{\exones}[6]{\includegraphics[width=.24\linewidth]{fig/ex1/s#1/i#2_pic.png} }

\begin{spacing}{9}
\noindent
\exones{4}{3}
{(0\,5\,13\,22\,14\,8\,4\,3)\,(1\,7\,17\,23\,18\,10\,6\,2)\,(9\,21\,19\,12)\,(11\,20\,15\,16)}
{(0\,1)\,(2\,3)\,(4\,5)\,(6\,7)\,(8\,9)\,(10\,11)\,(12\,13)\,(14\,15)\,(16\,17)\,(18\,19)\,(20\,21)\,(22\,23)}
{(0\,2\,4)\,(1\,3\,6)\,(5\,8\,12)\,(7\,10\,16)\,(9\,14\,20)\,(11\,18\,21)\,(13\,19\,23)\,(15\,22\,17)}
{$512 j^2 - 11734208 j + 6975534818043 = 0$}
\exones{4}{2}
{(0\,7\,15\,18\,12\,8\,5\,2\,1\,9\,17\,20\,14\,10\,4)\,(3\,11\,6)\,(13\,22\,16)\,(19\,21\,23)}
{(0\,1)\,(2\,3)\,(4\,5)\,(6\,7)\,(8\,9)\,(10\,11)\,(12\,13)\,(14\,15)\,(16\,17)\,(18\,19)\,(20\,21)\,(22\,23)}
{(0\,2\,6)\,(1\,4\,8)\,(3\,5\,10)\,(7\,11\,14)\,(9\,12\,16)\,(13\,18\,23)\,(15\,20\,19)\,(17\,22\,21)}
{$j = - \frac{20720464}{15625}$}
\exones{4}{1}
{(0\,7\,17\,14\,10\,4)\,(1\,9\,19\,22\,16\,13\,8\,5\,2)\,(3\,11\,21\,18\,12\,6)\,(15\,23\,20)}
{(0\,1)\,(2\,3)\,(4\,5)\,(6\,7)\,(8\,9)\,(10\,11)\,(12\,13)\,(14\,15)\,(16\,17)\,(18\,19)\,(20\,21)\,(22\,23)}
{(0\,2\,6)\,(1\,4\,8)\,(3\,5\,10)\,(7\,12\,16)\,(9\,13\,18)\,(11\,14\,20)\,(15\,17\,22)\,(19\,21\,23)}
{(4796 + 2872 I) z + (9995832 + 9478701 I)}
\exones{4}{4}
{(0\,5\,13\,18\,10\,6\,2\,1\,7\,17\,14\,8\,4\,3)\,(9\,20\,12)\,(11\,22\,16)\,(15\,23\,19\,21)}
{(0\,1)\,(2\,3)\,(4\,5)\,(6\,7)\,(8\,9)\,(10\,11)\,(12\,13)\,(14\,15)\,(16\,17)\,(18\,19)\,(20\,21)\,(22\,23)}
{(0\,2\,4)\,(1\,3\,6)\,(5\,8\,12)\,(7\,10\,16)\,(9\,14\,21)\,(11\,18\,23)\,(13\,20\,19)\,(15\,17\,22)}
{15193211036787 z^2 + -6177454460476992 z + -549860446486391685376}
\exones{4}{5}
{(0\,7\,16\,10\,4)\,(1\,9\,21\,22\,17\,13\,14\,8\,5\,2)\,(3\,11\,23\,18\,12\,6)\,(15\,19\,20)}
{(0\,1)\,(2\,3)\,(4\,5)\,(6\,7)\,(8\,9)\,(10\,11)\,(12\,13)\,(14\,15)\,(16\,17)\,(18\,19)\,(20\,21)\,(22\,23)}
{(0\,2\,6)\,(1\,4\,8)\,(3\,5\,10)\,(7\,12\,17)\,(9\,14\,20)\,(11\,16\,22)\,(13\,18\,15)\,(19\,23\,21)}
{z + 5000}
\exones{4}{6}
{(0\,7\,19\,14\,8\,5\,2\,1\,9\,16\,10\,4)\,(3\,11\,20\,12\,6)\,(13\,22\,18)\,(15\,23\,21\,17)}
{(0\,1)\,(2\,3)\,(4\,5)\,(6\,7)\,(8\,9)\,(10\,11)\,(12\,13)\,(14\,15)\,(16\,17)\,(18\,19)\,(20\,21)\,(22\,23)}
{(0\,2\,6)\,(1\,4\,8)\,(3\,5\,10)\,(7\,12\,18)\,(9\,14\,17)\,(11\,16\,21)\,(13\,20\,23)\,(15\,19\,22)}
{(30929104296875 + 19124209375000 I) z^2 + (-24523251599223168 + 1413449295656576 I) z + (-1642715590854894592 + 1507190869777049600 I)}
\exones{4}{7}
{(0\,7\,19\,15\,16\,10\,4)\,(1\,9\,18\,13\,22\,14\,8\,5\,2)\,(3\,11\,20\,12\,6)\,(17\,23\,21)}
{(0\,1)\,(2\,3)\,(4\,5)\,(6\,7)\,(8\,9)\,(10\,11)\,(12\,13)\,(14\,15)\,(16\,17)\,(18\,19)\,(20\,21)\,(22\,23)}
{(0\,2\,6)\,(1\,4\,8)\,(3\,5\,10)\,(7\,12\,18)\,(9\,14\,19)\,(11\,16\,21)\,(13\,20\,23)\,(15\,22\,17)}
{(3371174124589 - 229057872665 I) z^5 + (-11201152191024270 + 3325979797628043 I) z^4 + (-963338002220292 - 13989320370349287 I) z^3 + (17371569747049000 - 8103932165040713 I) z^2 + (-31746287436015590 + 4985400396450271 I) z + (-1696354221179599 - 4143980944541759 I)}
\exones{4}{8}
{(0\,5\,13\,19\,14\,8\,4\,3)\,(1\,7\,17\,22\,18\,10\,6\,2)\,(9\,21\,16\,11\,12)\,(15\,23\,20)}
{(0\,1)\,(2\,3)\,(4\,5)\,(6\,7)\,(8\,9)\,(10\,11)\,(12\,13)\,(14\,15)\,(16\,17)\,(18\,19)\,(20\,21)\,(22\,23)}
{(0\,2\,4)\,(1\,3\,6)\,(5\,8\,12)\,(7\,10\,16)\,(9\,14\,20)\,(11\,18\,13)\,(15\,19\,22)\,(17\,21\,23)}
{48828125 z + 4388755356576}
\exones{4}{9}
{(0\,5\,13\,14\,8\,4\,3)\,(1\,7\,17\,20\,12\,9\,18\,10\,6\,2)\,(11\,22\,16)\,(15\,21\,23\,19)}
{(0\,1)\,(2\,3)\,(4\,5)\,(6\,7)\,(8\,9)\,(10\,11)\,(12\,13)\,(14\,15)\,(16\,17)\,(18\,19)\,(20\,21)\,(22\,23)}
{(0\,2\,4)\,(1\,3\,6)\,(5\,8\,12)\,(7\,10\,16)\,(9\,14\,19)\,(11\,18\,23)\,(13\,20\,15)\,(17\,22\,21)}
{69328955078125 z^2 + -8525947821075000000 z + 42043929007034219215104}
\exones{4}{10}
{(0\,7\,14\,8\,5\,2\,1\,9\,21\,16\,10\,4)\,(3\,11\,18\,12\,6)\,(13\,23\,20\,15)\,(17\,22\,19)}
{(0\,1)\,(2\,3)\,(4\,5)\,(6\,7)\,(8\,9)\,(10\,11)\,(12\,13)\,(14\,15)\,(16\,17)\,(18\,19)\,(20\,21)\,(22\,23)}
{(0\,2\,6)\,(1\,4\,8)\,(3\,5\,10)\,(7\,12\,15)\,(9\,14\,20)\,(11\,16\,19)\,(13\,18\,22)\,(17\,21\,23)}
{(30929104296875 - 19124209375000 I) z^2 + (-24523251599223168 - 1413449295656576 I) z + (-1642715590854894592 - 1507190869777049600 I)}
\exones{4}{11}
{(0\,7\,17\,18\,12\,9\,5\,15\,22\,16\,10\,4)\,(1\,8\,2)\,(3\,13\,20\,14\,11\,6)\,(19\,23\,21)}
{(0\,1)\,(2\,3)\,(4\,5)\,(6\,7)\,(8\,9)\,(10\,11)\,(12\,13)\,(14\,15)\,(16\,17)\,(18\,19)\,(20\,21)\,(22\,23)}
{(0\,2\,6)\,(1\,4\,9)\,(3\,8\,12)\,(5\,10\,14)\,(7\,11\,16)\,(13\,18\,21)\,(15\,20\,23)\,(17\,22\,19)}
{1000000000 z^2 + -1776069315584 z + 88103303951473}
\exones{4}{12}
{(0\,7\,13\,20\,14\,10\,4)\,(1\,9\,6\,3\,11\,19\,16\,12\,8\,5\,2)\,(15\,22\,18)\,(17\,23\,21)}
{(0\,1)\,(2\,3)\,(4\,5)\,(6\,7)\,(8\,9)\,(10\,11)\,(12\,13)\,(14\,15)\,(16\,17)\,(18\,19)\,(20\,21)\,(22\,23)}
{(0\,2\,6)\,(1\,4\,8)\,(3\,5\,10)\,(7\,9\,12)\,(11\,14\,18)\,(13\,16\,21)\,(15\,20\,23)\,(17\,19\,22)}
{39135393 z + -47061251888}
\exones{4}{13}
{(0\,7\,16\,10\,4)\,(1\,9\,18\,12\,6\,3\,11\,20\,14\,8\,5\,2)\,(13\,23\,21\,17)\,(15\,22\,19)}
{(0\,1)\,(2\,3)\,(4\,5)\,(6\,7)\,(8\,9)\,(10\,11)\,(12\,13)\,(14\,15)\,(16\,17)\,(18\,19)\,(20\,21)\,(22\,23)}
{(0\,2\,6)\,(1\,4\,8)\,(3\,5\,10)\,(7\,12\,17)\,(9\,14\,19)\,(11\,16\,21)\,(13\,18\,22)\,(15\,20\,23)}
{10546875 z + -26410345352}
\exones{4}{14}
{(0\,7\,19\,16\,10\,4)\,(1\,9\,20\,12\,6\,3\,11\,14\,8\,5\,2)\,(13\,22\,18)\,(15\,17\,23\,21)}
{(0\,1)\,(2\,3)\,(4\,5)\,(6\,7)\,(8\,9)\,(10\,11)\,(12\,13)\,(14\,15)\,(16\,17)\,(18\,19)\,(20\,21)\,(22\,23)}
{(0\,2\,6)\,(1\,4\,8)\,(3\,5\,10)\,(7\,12\,18)\,(9\,14\,21)\,(11\,16\,15)\,(13\,20\,23)\,(17\,19\,22)}
{(263181652701 - 51613946582 I) z + (-497779137396288 + 41885412136432 I)}
\exones{4}{15}
{(0\,7\,19\,16\,10\,4)\,(1\,9\,12\,6\,3\,11\,20\,14\,8\,5\,2)\,(13\,15\,22\,18)\,(17\,23\,21)}
{(0\,1)\,(2\,3)\,(4\,5)\,(6\,7)\,(8\,9)\,(10\,11)\,(12\,13)\,(14\,15)\,(16\,17)\,(18\,19)\,(20\,21)\,(22\,23)}
{(0\,2\,6)\,(1\,4\,8)\,(3\,5\,10)\,(7\,12\,18)\,(9\,14\,13)\,(11\,16\,21)\,(15\,20\,23)\,(17\,19\,22)}
{(263181652701 + 51613946582 I) z + (-497779137396288 - 41885412136432 I)}
\exones{4}{16}
{(0\,7\,16\,10\,4)\,(1\,9\,21\,19\,14\,8\,5\,2)\,(3\,11\,23\,18\,12\,6)\,(13\,20\,15\,22\,17)}
{(0\,1)\,(2\,3)\,(4\,5)\,(6\,7)\,(8\,9)\,(10\,11)\,(12\,13)\,(14\,15)\,(16\,17)\,(18\,19)\,(20\,21)\,(22\,23)}
{(0\,2\,6)\,(1\,4\,8)\,(3\,5\,10)\,(7\,12\,17)\,(9\,14\,20)\,(11\,16\,22)\,(13\,18\,21)\,(15\,19\,23)}
{z + -23328}
\exones{4}{17}
{(0\,5\,9\,13\,20\,14\,10\,8\,6\,2\,1\,7\,4\,3)\,(11\,19\,16\,12)\,(15\,22\,18)\,(17\,23\,21)}
{(0\,1)\,(2\,3)\,(4\,5)\,(6\,7)\,(8\,9)\,(10\,11)\,(12\,13)\,(14\,15)\,(16\,17)\,(18\,19)\,(20\,21)\,(22\,23)}
{(0\,2\,4)\,(1\,3\,6)\,(5\,7\,8)\,(9\,10\,12)\,(11\,14\,18)\,(13\,16\,21)\,(15\,20\,23)\,(17\,19\,22)}
{55 z^6 + 24344 z^5 + 9424 z^4 + 25760 z^3 + -2743 z^2 + 20591 z + 6689}
\exones{4}{18}
{(0\,7\,19\,16\,10\,4)\,(1\,9\,23\,21\,17\,14\,8\,5\,2)\,(3\,11\,20\,12\,6)\,(13\,22\,15\,18)}
{(0\,1)\,(2\,3)\,(4\,5)\,(6\,7)\,(8\,9)\,(10\,11)\,(12\,13)\,(14\,15)\,(16\,17)\,(18\,19)\,(20\,21)\,(22\,23)}
{(0\,2\,6)\,(1\,4\,8)\,(3\,5\,10)\,(7\,12\,18)\,(9\,14\,22)\,(11\,16\,21)\,(13\,20\,23)\,(15\,17\,19)}
{z + (-864 + 4752 I)}
%
%
\exones{4}{20}
{(0\,7\,21\,16\,10\,5\,19\,23\,20\,12\,4)\,(1\,11\,8\,2)\,(3\,15\,18\,13\,6)\,(9\,17\,22\,14)}
{(0\,1)\,(2\,3)\,(4\,5)\,(6\,7)\,(8\,9)\,(10\,11)\,(12\,13)\,(14\,15)\,(16\,17)\,(18\,19)\,(20\,21)\,(22\,23)}
{(0\,2\,6)\,(1\,4\,10)\,(3\,8\,14)\,(5\,12\,18)\,(7\,13\,20)\,(9\,11\,16)\,(15\,22\,19)\,(17\,21\,23)}
{5153632 z + -6761990971}
%
%
\exones{4}{22}
{(0\,7\,19\,14\,8\,5\,2\,1\,9\,21\,16\,10\,4)\,(3\,11\,12\,6)\,(13\,17\,22\,18)\,(15\,23\,20)}
{(0\,1)\,(2\,3)\,(4\,5)\,(6\,7)\,(8\,9)\,(10\,11)\,(12\,13)\,(14\,15)\,(16\,17)\,(18\,19)\,(20\,21)\,(22\,23)}
{(0\,2\,6)\,(1\,4\,8)\,(3\,5\,10)\,(7\,12\,18)\,(9\,14\,20)\,(11\,16\,13)\,(15\,19\,22)\,(17\,21\,23)}
{2568640535 z^7 + 677876061660 z^6 + -407807704440 z^5 + 529883815003 z^4 + -1280941806089 z^3 + -210833521054 z^2 + -529643138567 z + 106064986500}
\exones{4}{23}
{(0\,7\,18\,10\,4)\,(1\,8\,2)\,(3\,15\,22\,19\,13\,20\,14\,9\,5\,17\,12\,6)\,(11\,23\,21\,16)}
{(0\,1)\,(2\,3)\,(4\,5)\,(6\,7)\,(8\,9)\,(10\,11)\,(12\,13)\,(14\,15)\,(16\,17)\,(18\,19)\,(20\,21)\,(22\,23)}
{(0\,2\,6)\,(1\,4\,9)\,(3\,8\,14)\,(5\,10\,16)\,(7\,12\,19)\,(11\,18\,22)\,(13\,17\,21)\,(15\,20\,23)}
{(30929104296875 - 19124209375000 I) z^2 + (-24523251599223168 - 1413449295656576 I) z + (-1642715590854894592 - 1507190869777049600 I)}
\exones{4}{24}
{(0 7 8 5 2 1 9 13 20 14 10 4) (3 11 19 16 12 6) (15 22 18) (17 23 21)}
{(0 1) (2 3) (4 5) (6 7) (8 9) (10 11) (12 13) (14 15) (16 17) (18 19) (20 21) (22 23)}
{(0 2 6) (1 4 8) (3 5 10) (7 12 9) (11 14 18) (13 16 21) (15 20 23) (17 19 22)}
{}
\exones{4}{25}
{(0\,7\,23\,17\,19\,14\,9\,22\,12\,4)\,(1\,11\,16\,8\,2)\,(3\,15\,20\,13\,6)\,(5\,21\,18\,10)}
{(0\,1)\,(2\,3)\,(4\,5)\,(6\,7)\,(8\,9)\,(10\,11)\,(12\,13)\,(14\,15)\,(16\,17)\,(18\,19)\,(20\,21)\,(22\,23)}
{(0\,2\,6)\,(1\,4\,10)\,(3\,8\,14)\,(5\,12\,20)\,(7\,13\,22)\,(9\,16\,23)\,(11\,18\,17)\,(15\,19\,21)}
{(117 + 44 I) z + (-37408 + 82192 I)}
\exones{4}{26}
{(0\,5\,13\,14\,8\,4\,3)\,(1\,7\,17\,18\,10\,6\,2)\,(9\,23\,19\,20\,12)\,(11\,22\,15\,21\,16)}
{(0\,1)\,(2\,3)\,(4\,5)\,(6\,7)\,(8\,9)\,(10\,11)\,(12\,13)\,(14\,15)\,(16\,17)\,(18\,19)\,(20\,21)\,(22\,23)}
{(0\,2\,4)\,(1\,3\,6)\,(5\,8\,12)\,(7\,10\,16)\,(9\,14\,22)\,(11\,18\,23)\,(13\,20\,15)\,(17\,21\,19)}
{z^2 + -914416 z + 590816592}
\exones{4}{27}
{(0\,7\,19\,20\,14\,9\,5\,12\,6\,3\,15\,16\,10\,4)\,(1\,8\,2)\,(11\,23\,18\,13)\,(17\,21\,22)}
{(0\,1)\,(2\,3)\,(4\,5)\,(6\,7)\,(8\,9)\,(10\,11)\,(12\,13)\,(14\,15)\,(16\,17)\,(18\,19)\,(20\,21)\,(22\,23)}
{(0\,2\,6)\,(1\,4\,9)\,(3\,8\,14)\,(5\,10\,13)\,(7\,12\,18)\,(11\,16\,22)\,(15\,20\,17)\,(19\,23\,21)}
{3769304304267952128 z^2 + 3284219076931720839168 z + 69206969233824288109873}
\exones{4}{28}
{(0\,7\,21\,16\,8\,2\,1\,11\,22\,14\,9\,4)\,(3\,15\,12\,6)\,(5\,17\,18\,10)\,(13\,23\,19\,20)}
{(0\,1)\,(2\,3)\,(4\,5)\,(6\,7)\,(8\,9)\,(10\,11)\,(12\,13)\,(14\,15)\,(16\,17)\,(18\,19)\,(20\,21)\,(22\,23)}
{(0\,2\,6)\,(1\,4\,10)\,(3\,8\,14)\,(5\,9\,16)\,(7\,12\,20)\,(11\,18\,23)\,(13\,15\,22)\,(17\,21\,19)}
{6561 z + -207646}
\exones{4}{29}
{(0\,7\,21\,22\,14\,9\,17\,20\,12\,4)\,(1\,11\,8\,2)\,(3\,15\,18\,13\,6)\,(5\,19\,23\,16\,10)}
{(0\,1)\,(2\,3)\,(4\,5)\,(6\,7)\,(8\,9)\,(10\,11)\,(12\,13)\,(14\,15)\,(16\,17)\,(18\,19)\,(20\,21)\,(22\,23)}
{(0\,2\,6)\,(1\,4\,10)\,(3\,8\,14)\,(5\,12\,18)\,(7\,13\,20)\,(9\,11\,16)\,(15\,22\,19)\,(17\,23\,21)}
{(117 - 44 I) z + (-37408 - 82192 I)}
\exones{4}{30}
{(0\,7\,19\,16\,10\,4)\,(1\,8\,2)\,(3\,15\,22\,18\,13\,11\,20\,14\,9\,5\,12\,6)\,(17\,23\,21)}
{(0\,1)\,(2\,3)\,(4\,5)\,(6\,7)\,(8\,9)\,(10\,11)\,(12\,13)\,(14\,15)\,(16\,17)\,(18\,19)\,(20\,21)\,(22\,23)}
{(0\,2\,6)\,(1\,4\,9)\,(3\,8\,14)\,(5\,10\,13)\,(7\,12\,18)\,(11\,16\,21)\,(15\,20\,23)\,(17\,19\,22)}
{z^2 + -869824 z + 1840398592}
\exones{4}{31}
{(0\,7\,19\,22\,16\,10\,4)\,(1\,9\,21\,18\,13\,17\,14\,8\,5\,2)\,(3\,11\,12\,6)\,(15\,23\,20)}
{(0\,1)\,(2\,3)\,(4\,5)\,(6\,7)\,(8\,9)\,(10\,11)\,(12\,13)\,(14\,15)\,(16\,17)\,(18\,19)\,(20\,21)\,(22\,23)}
{(0\,2\,6)\,(1\,4\,8)\,(3\,5\,10)\,(7\,12\,18)\,(9\,14\,20)\,(11\,16\,13)\,(15\,17\,22)\,(19\,21\,23)}
{69328955078125 z^2 + -8525947821075000000 z + 42043929007034219215104}
\exones{4}{32}
{(0\,5\,13\,8\,4\,3)\,(1\,7\,15\,22\,16\,10\,6\,2)\,(9\,19\,23\,21\,17\,18\,12)\,(11\,20\,14)}
{(0\,1)\,(2\,3)\,(4\,5)\,(6\,7)\,(8\,9)\,(10\,11)\,(12\,13)\,(14\,15)\,(16\,17)\,(18\,19)\,(20\,21)\,(22\,23)}
{(0\,2\,4)\,(1\,3\,6)\,(5\,8\,12)\,(7\,10\,14)\,(9\,13\,18)\,(11\,16\,21)\,(15\,20\,23)\,(17\,22\,19)}
{1953125 z + 26240594230368}
\exones{4}{33}
{(0\,7\,19\,16\,11\,21\,23\,18\,13\,10\,4)\,(1\,8\,2)\,(3\,15\,20\,12\,6)\,(5\,17\,22\,14\,9)}
{(0\,1)\,(2\,3)\,(4\,5)\,(6\,7)\,(8\,9)\,(10\,11)\,(12\,13)\,(14\,15)\,(16\,17)\,(18\,19)\,(20\,21)\,(22\,23)}
{(0\,2\,6)\,(1\,4\,9)\,(3\,8\,14)\,(5\,10\,16)\,(7\,12\,18)\,(11\,13\,20)\,(15\,22\,21)\,(17\,19\,23)}
{161051 z + -54010152}
\exones{4}{34}
{(0\,7\,17\,20\,14\,11\,6\,3\,13\,22\,16\,10\,4)\,(1\,8\,2)\,(5\,15\,18\,12\,9)\,(19\,21\,23)}
{(0\,1)\,(2\,3)\,(4\,5)\,(6\,7)\,(8\,9)\,(10\,11)\,(12\,13)\,(14\,15)\,(16\,17)\,(18\,19)\,(20\,21)\,(22\,23)}
{(0\,2\,6)\,(1\,4\,9)\,(3\,8\,12)\,(5\,10\,14)\,(7\,11\,16)\,(13\,18\,23)\,(15\,20\,19)\,(17\,22\,21)}
{996144 z^10 + -16552808 z^9 + -195320864 z^8 + -112543706 z^7 + -99441532 z^6 + 176354363 z^5 + -46302054 z^4 + -232550554 z^3 + 146891750 z^2 + -106334632 z + 103684110}
\exones{4}{35}
{(0\,5\,13\,16\,10\,6\,2\,1\,7\,14\,8\,4\,3)\,(9\,21\,18\,12)\,(11\,23\,20\,15)\,(17\,19\,22)}
{(0\,1)\,(2\,3)\,(4\,5)\,(6\,7)\,(8\,9)\,(10\,11)\,(12\,13)\,(14\,15)\,(16\,17)\,(18\,19)\,(20\,21)\,(22\,23)}
{(0\,2\,4)\,(1\,3\,6)\,(5\,8\,12)\,(7\,10\,15)\,(9\,14\,20)\,(11\,16\,22)\,(13\,18\,17)\,(19\,21\,23)}
{1410568077862 z^5 + 5539572014145901 z^4 + -2891744982321219 z^3 + 12985298732686739 z^2 + -13643185382061004 z + 3258853360771286}
\exones{4}{36}
{(0\,5\,13\,22\,14\,8\,4\,3)\,(1\,7\,17\,20\,15\,10\,6\,2)\,(9\,21\,18\,12)\,(11\,23\,19\,16)}
{(0\,1)\,(2\,3)\,(4\,5)\,(6\,7)\,(8\,9)\,(10\,11)\,(12\,13)\,(14\,15)\,(16\,17)\,(18\,19)\,(20\,21)\,(22\,23)}
{(0\,2\,4)\,(1\,3\,6)\,(5\,8\,12)\,(7\,10\,16)\,(9\,14\,20)\,(11\,15\,22)\,(13\,18\,23)\,(17\,19\,21)}
{512 z^2 + -11734208 z + 6975534818043}
\exones{4}{37}
{(0\,7\,16\,10\,4)\,(1\,9\,19\,20\,14\,8\,5\,2)\,(3\,11\,23\,18\,15\,12\,6)\,(13\,21\,22\,17)}
{(0\,1)\,(2\,3)\,(4\,5)\,(6\,7)\,(8\,9)\,(10\,11)\,(12\,13)\,(14\,15)\,(16\,17)\,(18\,19)\,(20\,21)\,(22\,23)}
{(0\,2\,6)\,(1\,4\,8)\,(3\,5\,10)\,(7\,12\,17)\,(9\,14\,18)\,(11\,16\,22)\,(13\,15\,20)\,(19\,23\,21)}
{1977326743 z + 30460567656096}
\exones{4}{38}
{(0\,5\,13\,11\,20\,14\,8\,4\,3)\,(1\,7\,12\,9\,19\,16\,10\,6\,2)\,(15\,22\,18)\,(17\,23\,21)}
{(0\,1)\,(2\,3)\,(4\,5)\,(6\,7)\,(8\,9)\,(10\,11)\,(12\,13)\,(14\,15)\,(16\,17)\,(18\,19)\,(20\,21)\,(22\,23)}
{(0\,2\,4)\,(1\,3\,6)\,(5\,8\,12)\,(7\,10\,13)\,(9\,14\,18)\,(11\,16\,21)\,(15\,20\,23)\,(17\,19\,22)}
{531441 z^3 + -42488333973 z^2 + 366016695860736 z + -13234833427661329552}
\exones{4}{39}
{(0\,7\,21\,23\,18\,13\,6\,3\,15\,20\,12\,4)\,(1\,11\,8\,2)\,(5\,19\,16\,10)\,(9\,17\,22\,14)}
{(0\,1)\,(2\,3)\,(4\,5)\,(6\,7)\,(8\,9)\,(10\,11)\,(12\,13)\,(14\,15)\,(16\,17)\,(18\,19)\,(20\,21)\,(22\,23)}
{(0\,2\,6)\,(1\,4\,10)\,(3\,8\,14)\,(5\,12\,18)\,(7\,13\,20)\,(9\,11\,16)\,(15\,22\,21)\,(17\,19\,23)}
{z}
\exones{4}{40}
{(0\,7\,15\,23\,21\,19\,14\,10\,4)\,(1\,9\,17\,22\,18\,12\,8\,5\,2)\,(3\,11\,6)\,(13\,20\,16)}
{(0\,1)\,(2\,3)\,(4\,5)\,(6\,7)\,(8\,9)\,(10\,11)\,(12\,13)\,(14\,15)\,(16\,17)\,(18\,19)\,(20\,21)\,(22\,23)}
{(0\,2\,6)\,(1\,4\,8)\,(3\,5\,10)\,(7\,11\,14)\,(9\,12\,16)\,(13\,18\,21)\,(15\,19\,22)\,(17\,20\,23)}
{z + -54000}
\exones{4}{41}
{(0\,5\,13\,20\,14\,8\,4\,3)\,(1\,7\,17\,12\,9\,18\,10\,6\,2)\,(11\,23\,21\,16)\,(15\,22\,19)}
{(0\,1)\,(2\,3)\,(4\,5)\,(6\,7)\,(8\,9)\,(10\,11)\,(12\,13)\,(14\,15)\,(16\,17)\,(18\,19)\,(20\,21)\,(22\,23)}
{(0\,2\,4)\,(1\,3\,6)\,(5\,8\,12)\,(7\,10\,16)\,(9\,14\,19)\,(11\,18\,22)\,(13\,17\,21)\,(15\,20\,23)}
{(2298884260292 + 100012551538 I) z^5 + (-100758956854945559 + 44010666654603979 I) z^4 + (160873343514275743 + 54917137216946258 I) z^3 + (-51311348098831066 + 133483663927740121 I) z^2 + (-195017859998852070 + 120676342295702809 I) z + (79172282236877200 + 62209810439439668 I)}
\exones{4}{42}
{(0\,7\,19\,22\,14\,9\,5\,17\,18\,13\,10\,4)\,(1\,8\,2)\,(3\,15\,20\,12\,6)\,(11\,21\,23\,16)}
{(0\,1)\,(2\,3)\,(4\,5)\,(6\,7)\,(8\,9)\,(10\,11)\,(12\,13)\,(14\,15)\,(16\,17)\,(18\,19)\,(20\,21)\,(22\,23)}
{(0\,2\,6)\,(1\,4\,9)\,(3\,8\,14)\,(5\,10\,16)\,(7\,12\,18)\,(11\,13\,20)\,(15\,22\,21)\,(17\,23\,19)}
{(30929104296875 + 19124209375000 I) z^2 + (-24523251599223168 + 1413449295656576 I) z + (-1642715590854894592 + 1507190869777049600 I)}
\exones{4}{43}
{(0\,5\,13\,14\,8\,4\,3)\,(1\,7\,17\,18\,10\,6\,2)\,(9\,23\,16\,11\,20\,12)\,(15\,21\,19\,22)}
{(0\,1)\,(2\,3)\,(4\,5)\,(6\,7)\,(8\,9)\,(10\,11)\,(12\,13)\,(14\,15)\,(16\,17)\,(18\,19)\,(20\,21)\,(22\,23)}
{(0\,2\,4)\,(1\,3\,6)\,(5\,8\,12)\,(7\,10\,16)\,(9\,14\,22)\,(11\,18\,21)\,(13\,20\,15)\,(17\,23\,19)}
{4782969 z^2 + -140664190437568 z + -107680204166497530624}
\exones{4}{44}
{(0\,7\,19\,16\,10\,4)\,(1\,9\,21\,17\,22\,14\,8\,5\,2)\,(3\,11\,20\,15\,12\,6)\,(13\,23\,18)}
{(0\,1)\,(2\,3)\,(4\,5)\,(6\,7)\,(8\,9)\,(10\,11)\,(12\,13)\,(14\,15)\,(16\,17)\,(18\,19)\,(20\,21)\,(22\,23)}
{(0\,2\,6)\,(1\,4\,8)\,(3\,5\,10)\,(7\,12\,18)\,(9\,14\,20)\,(11\,16\,21)\,(13\,15\,22)\,(17\,19\,23)}
{(4796 - 2872 I) z + (9995832 - 9478701 I)}
\exones{4}{45}
{(0\,7\,19\,20\,15\,16\,10\,4)\,(1\,9\,21\,22\,14\,8\,5\,2)\,(3\,11\,12\,6)\,(13\,17\,23\,18)}
{(0\,1)\,(2\,3)\,(4\,5)\,(6\,7)\,(8\,9)\,(10\,11)\,(12\,13)\,(14\,15)\,(16\,17)\,(18\,19)\,(20\,21)\,(22\,23)}
{(0\,2\,6)\,(1\,4\,8)\,(3\,5\,10)\,(7\,12\,18)\,(9\,14\,20)\,(11\,16\,13)\,(15\,22\,17)\,(19\,23\,21)}
{z + -10976}
\exones{4}{46}
{(0\,5\,13\,14\,8\,4\,3)\,(1\,7\,17\,20\,15\,18\,10\,6\,2)\,(9\,21\,23\,19\,12)\,(11\,22\,16)}
{(0\,1)\,(2\,3)\,(4\,5)\,(6\,7)\,(8\,9)\,(10\,11)\,(12\,13)\,(14\,15)\,(16\,17)\,(18\,19)\,(20\,21)\,(22\,23)}
{(0\,2\,4)\,(1\,3\,6)\,(5\,8\,12)\,(7\,10\,16)\,(9\,14\,20)\,(11\,18\,23)\,(13\,19\,15)\,(17\,22\,21)}
{669845673459105 z^4 + 155788641386171663148 z^3 + 75138411067832374629 z^2 + 11526818657062076735 z + -304118614767026389015}
\exones{4}{47}
{(0\,7\,10\,4)\,(1\,9\,19\,16\,12\,6\,3\,11\,13\,20\,14\,8\,5\,2)\,(15\,22\,18)\,(17\,23\,21)}
{(0\,1)\,(2\,3)\,(4\,5)\,(6\,7)\,(8\,9)\,(10\,11)\,(12\,13)\,(14\,15)\,(16\,17)\,(18\,19)\,(20\,21)\,(22\,23)}
{(0\,2\,6)\,(1\,4\,8)\,(3\,5\,10)\,(7\,12\,11)\,(9\,14\,18)\,(13\,16\,21)\,(15\,20\,23)\,(17\,19\,22)}
{15193211036787 z^2 + -6177454460476992 z + -549860446486391685376}
\exones{4}{48}
{(0\,5\,13\,20\,14\,11\,8\,4\,3)\,(1\,7\,15\,22\,16\,10\,6\,2)\,(9\,17\,18\,12)\,(19\,23\,21)}
{(0\,1)\,(2\,3)\,(4\,5)\,(6\,7)\,(8\,9)\,(10\,11)\,(12\,13)\,(14\,15)\,(16\,17)\,(18\,19)\,(20\,21)\,(22\,23)}
{(0\,2\,4)\,(1\,3\,6)\,(5\,8\,12)\,(7\,10\,14)\,(9\,11\,16)\,(13\,18\,21)\,(15\,20\,23)\,(17\,22\,19)}
{(1007830715797 + 956523396018 I) z^5 + (-23152751333306184 - 62228049168901629 I) z^4 + (2345131608540570 + 138240885331428878 I) z^3 + (20522021106106053 + 56791716145333504 I) z^2 + (144600801530787277 - 35136929200177764 I) z + (371860448190474181 - 10001114131898457 I)}
\exones{4}{49}
{(0\,5\,13\,8\,4\,3)\,(1\,7\,15\,16\,10\,6\,2)\,(9\,19\,22\,17\,21\,18\,12)\,(11\,23\,20\,14)}
{(0\,1)\,(2\,3)\,(4\,5)\,(6\,7)\,(8\,9)\,(10\,11)\,(12\,13)\,(14\,15)\,(16\,17)\,(18\,19)\,(20\,21)\,(22\,23)}
{(0\,2\,4)\,(1\,3\,6)\,(5\,8\,12)\,(7\,10\,14)\,(9\,13\,18)\,(11\,16\,22)\,(15\,20\,17)\,(19\,21\,23)}
{4782969 z^2 + -140664190437568 z + -107680204166497530624}
\exones{4}{50}
{(0\,7\,21\,23\,19\,14\,9\,4)\,(1\,11\,22\,16\,8\,2)\,(3\,15\,12\,6)\,(5\,17\,20\,13\,18\,10)}
{(0\,1)\,(2\,3)\,(4\,5)\,(6\,7)\,(8\,9)\,(10\,11)\,(12\,13)\,(14\,15)\,(16\,17)\,(18\,19)\,(20\,21)\,(22\,23)}
{(0\,2\,6)\,(1\,4\,10)\,(3\,8\,14)\,(5\,9\,16)\,(7\,12\,20)\,(11\,18\,23)\,(13\,15\,19)\,(17\,22\,21)}
{729 z + 219488}
\exones{4}{51}
{(0\,5\,10\,6\,2\,1\,7\,15\,18\,12\,8\,4\,3)\,(9\,17\,20\,14\,11)\,(13\,22\,16)\,(19\,21\,23)}
{(0\,1)\,(2\,3)\,(4\,5)\,(6\,7)\,(8\,9)\,(10\,11)\,(12\,13)\,(14\,15)\,(16\,17)\,(18\,19)\,(20\,21)\,(22\,23)}
{(0\,2\,4)\,(1\,3\,6)\,(5\,8\,11)\,(7\,10\,14)\,(9\,12\,16)\,(13\,18\,23)\,(15\,20\,19)\,(17\,22\,21)}
{23899969500 z^6 + 61261757543144 z^5 + -23028739753252 z^4 + -65497515974681 z^3 + -123340630327773 z^2 + 63898121291643 z + 37992815970307}
\exones{4}{52}
{(0\,7\,15\,22\,19\,20\,14\,10\,4)\,(1\,9\,17\,18\,12\,8\,5\,2)\,(3\,11\,6)\,(13\,23\,21\,16)}
{(0\,1)\,(2\,3)\,(4\,5)\,(6\,7)\,(8\,9)\,(10\,11)\,(12\,13)\,(14\,15)\,(16\,17)\,(18\,19)\,(20\,21)\,(22\,23)}
{(0\,2\,6)\,(1\,4\,8)\,(3\,5\,10)\,(7\,11\,14)\,(9\,12\,16)\,(13\,18\,22)\,(15\,20\,23)\,(17\,21\,19)}
{989315216267943 z^4 + 120549271573764602075 z^3 + 15102273256307590681 z^2 + -59850321618912269637 z + 9729207218973661900}
\exones{4}{53}
{(0\,5\,11\,19\,16\,12\,8\,4\,3)\,(1\,7\,13\,20\,14\,10\,9\,6\,2)\,(15\,22\,18)\,(17\,23\,21)}
{(0\,1)\,(2\,3)\,(4\,5)\,(6\,7)\,(8\,9)\,(10\,11)\,(12\,13)\,(14\,15)\,(16\,17)\,(18\,19)\,(20\,21)\,(22\,23)}
{(0\,2\,4)\,(1\,3\,6)\,(5\,8\,10)\,(7\,9\,12)\,(11\,14\,18)\,(13\,16\,21)\,(15\,20\,23)\,(17\,19\,22)}
{531441 z^3 + -42488333973 z^2 + 366016695860736 z + -13234833427661329552}
\exones{4}{54}
{(0\,7\,18\,10\,5\,17\,20\,14\,9\,4)\,(1\,11\,22\,16\,8\,2)\,(3\,15\,12\,6)\,(13\,21\,23\,19)}
{(0\,1)\,(2\,3)\,(4\,5)\,(6\,7)\,(8\,9)\,(10\,11)\,(12\,13)\,(14\,15)\,(16\,17)\,(18\,19)\,(20\,21)\,(22\,23)}
{(0\,2\,6)\,(1\,4\,10)\,(3\,8\,14)\,(5\,9\,16)\,(7\,12\,19)\,(11\,18\,23)\,(13\,15\,20)\,(17\,22\,21)}
{31250 z + -57960603}
\exones{4}{55}
{(0\,5\,13\,8\,4\,3)\,(1\,7\,15\,22\,18\,12\,9\,19\,16\,10\,6\,2)\,(11\,20\,14)\,(17\,23\,21)}
{(0\,1)\,(2\,3)\,(4\,5)\,(6\,7)\,(8\,9)\,(10\,11)\,(12\,13)\,(14\,15)\,(16\,17)\,(18\,19)\,(20\,21)\,(22\,23)}
{(0\,2\,4)\,(1\,3\,6)\,(5\,8\,12)\,(7\,10\,14)\,(9\,13\,18)\,(11\,16\,21)\,(15\,20\,23)\,(17\,19\,22)}
{z^2 + -869824 z + 1840398592}
\exones{4}{56}
{(0\,7\,17\,20\,15\,10\,4)\,(1\,9\,21\,23\,19\,14\,8\,5\,2)\,(3\,11\,18\,12\,6)\,(13\,22\,16)}
{(0\,1)\,(2\,3)\,(4\,5)\,(6\,7)\,(8\,9)\,(10\,11)\,(12\,13)\,(14\,15)\,(16\,17)\,(18\,19)\,(20\,21)\,(22\,23)}
{(0\,2\,6)\,(1\,4\,8)\,(3\,5\,10)\,(7\,12\,16)\,(9\,14\,20)\,(11\,15\,19)\,(13\,18\,23)\,(17\,22\,21)}
{(3371174124589 + 229057872665 I) z^5 + (-11201152191024270 - 3325979797628043 I) z^4 + (-963338002220292 + 13989320370349287 I) z^3 + (17371569747049000 + 8103932165040713 I) z^2 + (-31746287436015590 - 4985400396450271 I) z + (-1696354221179599 + 4143980944541759 I)}
%
%
\exones{4}{58}
{(0\,5\,13\,23\,20\,14\,8\,4\,3)\,(1\,7\,17\,22\,19\,15\,10\,6\,2)\,(9\,18\,12)\,(11\,21\,16)}
{(0\,1)\,(2\,3)\,(4\,5)\,(6\,7)\,(8\,9)\,(10\,11)\,(12\,13)\,(14\,15)\,(16\,17)\,(18\,19)\,(20\,21)\,(22\,23)}
{(0\,2\,4)\,(1\,3\,6)\,(5\,8\,12)\,(7\,10\,16)\,(9\,14\,19)\,(11\,15\,20)\,(13\,18\,22)\,(17\,21\,23)}
{531441 z^3 + -42488333973 z^2 + 366016695860736 z + -13234833427661329552}
\exones{4}{60}
{(0\,7\,15\,22\,16\,13\,20\,14\,10\,4)\,(1\,9\,17\,18\,12\,8\,5\,2)\,(3\,11\,6)\,(19\,23\,21)}
{(0\,1)\,(2\,3)\,(4\,5)\,(6\,7)\,(8\,9)\,(10\,11)\,(12\,13)\,(14\,15)\,(16\,17)\,(18\,19)\,(20\,21)\,(22\,23)}
{(0\,2\,6)\,(1\,4\,8)\,(3\,5\,10)\,(7\,11\,14)\,(9\,12\,16)\,(13\,18\,21)\,(15\,20\,23)\,(17\,22\,19)}
{3796875 z + -192596360288}
\exones{4}{61}
{(0\,7\,21\,23\,18\,10\,4)\,(1\,8\,2)\,(3\,15\,22\,16\,11\,12\,6)\,(5\,17\,20\,13\,19\,14\,9)}
{(0\,1)\,(2\,3)\,(4\,5)\,(6\,7)\,(8\,9)\,(10\,11)\,(12\,13)\,(14\,15)\,(16\,17)\,(18\,19)\,(20\,21)\,(22\,23)}
{(0\,2\,6)\,(1\,4\,9)\,(3\,8\,14)\,(5\,10\,16)\,(7\,12\,20)\,(11\,18\,13)\,(15\,19\,23)\,(17\,22\,21)}
{z}
%
%
%

\end{spacing}

\bibliographystyle{siam}
\bibliography{Biblio}
\end{document}